\documentclass{amsart}
\usepackage{amssymb}
\usepackage{amsfonts}
\usepackage{amsmath}
\usepackage{hyperref}
\usepackage{geometry}
\usepackage{xcolor}

\newtheorem{theorem}{Theorem}
\theoremstyle{plain}

\newtheorem{definition}{Definition}

\newtheorem{lemma}{Lemma}
\newtheorem{notation}{Notation}

\newtheorem{proposition}{Proposition}
\newtheorem{remark}{Remark}

\numberwithin{equation}{section}
\input{tcilatex}
\begin{document}
\title{Continuity of infinitely degenerate weak solutions via the trace
method}
\author{Lyudmila Korobenko}
\author{Eric T. Sawyer}
\maketitle

\begin{abstract}
In 1971 Fedi\u{\i} proved in \cite{Fe} the remarkable theorem that the
linear second order partial differential operator%
\begin{equation*}
L_{f}u\left( x,y\right) \equiv \left\{ \frac{\partial }{\partial x^{2}}%
+f\left( x\right) ^{2}\frac{\partial }{\partial y^{2}}\right\} u\left(
x,y\right) 
\end{equation*}%
is \emph{hypoelliptic} provided that $f\in C^{\infty }\left( \mathbb{R}%
\right) $, $f\left( 0\right) =0$ and $f$ is positive on $\left( -\infty
,0\right) \cup \left( 0,\infty \right) $. Variants of this result, with
hypoellipticity replaced by continuity of weak solutions, were recently
given by the authors, together with Cristian Rios and Ruipeng Shen, in \cite%
{KoRiSaSh} to infinitely degenerate elliptic divergence form equations%
\begin{equation*}
\nabla ^{\limfunc{tr}}\mathcal{A}\left( x,u\right) \nabla u=\phi \left(
x\right) ,\ \ \ \ \ x\in \Omega \subset \mathbb{R}^{n},
\end{equation*}%
where the nonnegative matrix $\mathcal{A}\left( x,u\right) $ has \emph{%
bounded measurable coefficients} with trace roughly $1$ and determinant
comparable to $f^{2}$, and where $F=\ln \frac{1}{f}$ is essentially doubling.

However, in the plane, these variants assumed\ additional geometric
constraints on $f$, such as $f\left( r\right) \geq e^{-r^{-\sigma }}$ for
some $0<\sigma <1$, something not required in Fedi\u{\i}'s theorem. In this
paper we in particular remove these additional geometric constraints in the
plane for homogeneous equations with $F$ essentially doubling.
\end{abstract}

\tableofcontents

\section{Introduction}

In 1971 Fedi\u{\i} proved in \cite{Fe} the remarkable theorem that the
linear second order partial differential operator%
\begin{equation*}
L_{f}u\left( x,y\right) \equiv \left\{ \frac{\partial }{\partial x^{2}}%
+f\left( x\right) ^{2}\frac{\partial }{\partial y^{2}}\right\} u\left(
x,y\right)
\end{equation*}%
is \emph{hypoelliptic}, i.e. every distribution solution $u\in \mathcal{D}%
^{\prime }\left( \mathbb{R}^{2}\right) $ to the equation $L_{f}u=\phi \in
C^{\infty }\left( \mathbb{R}^{2}\right) $ in $\mathbb{R}^{2}$ is smooth,
i.e. $u\in C^{\infty }\left( \mathbb{R}^{2}\right) $, provided that $f\in
C^{\infty }\left( \mathbb{R}\right) $, $f\left( 0\right) =0$ and $f$ is
positive on $\left( -\infty ,0\right) \cup \left( 0,\infty \right) $. In
particular $f$ can vanish to \emph{infinite} order and $L_{f}$ is infinitely
degenerate elliptic. See also \cite{KuStr}, \cite{Mor}, \cite{Chr} and \cite%
{Koh} for generalizations to \emph{smooth} equations in higher dimensions,
something we do not pursue here.

Variants of this result were then given in \cite{KoRiSaSh} to infinitely
degenerate elliptic divergence form equations%
\begin{equation*}
\nabla ^{\limfunc{tr}}A\left( x\right) \nabla u=\phi \left( x\right) ,\ \ \
\ \ x\in \Omega \subset \mathbb{R}^{n},
\end{equation*}%
where the nonnegative matrix $A$ has \emph{bounded measurable coefficients}
with trace roughly $1$ and determinant comparable to $f^{2}$. The concept of
hypoellipticity was interpreted there in terms of local boundedness and
continuity of weak solutions. However, additional geometric constraints on
the degeneracy $f$ were needed for the methods used there, beyond the
minimal restriction that $F\equiv \ln \frac{1}{f}$ be a `structured
geometry', i.e. satisfies the five `log doubling' structure conditions in
Definition \ref{structure conditions} below (useful for estimating arc
length and volume of control balls). While these additional geometric
constraints were often necessary in dimension $n\geq 3$, as is the case for
the smooth equations in higher dimensions in \cite{KuStr} and \cite{Chr},
the case of dimension $n=2$ was left open in the rough setting.

The main goal of this paper is to extend this type of result to the plane $%
\mathbb{R}^{2}$ for all structured geometries without \emph{any additional
geometric constraints}. More precisely, to certain equations in the plane $%
\mathbb{R}^{2}$ with nonegative matrices having bounded measurable
coefficients with trace roughly$\ 1$ and determinant $f^{2}$, and with at
least bounded forcing functions $\phi $. For this purpose we develop a \emph{%
trace method} that first constructs a region in $\mathbb{R}^{2}$ on whose
boundary a given subsolution $u$ has a suitable trace, and second applies a
maximum principle to derive local boundedness and continuity of weak
solutions $u$ to some infinitely degenerate equations, resulting in our
Trace Method Theorem.

We will use both an existing maximum principle for inhomogeneous equations
from \cite{KoRiSaSh}, that requires a restriction on the geometry $F$
associated with the operator, as well as a new maximum principle for
homogeneous equations, valid for all structured geometries, and all
dimensions as well. Table 1 organizes the various conclusions on weak
solutions so that they either persist or improve as we move to the right or
lower down in the table.

There are three separate notions of admissibilty of inhomogeneous data $\phi 
$ appearing in this table: the strongest notion is that\ used in \cite%
{KoRiSaSh2} which amounts to assuming the data are very close to $L^{\infty
} $; the weakest notion is that in \cite{KoRiSaSh}, denoted $\phi \in X_{f}$
here; and an intermediate notion, called $f$-admissible, that is used in the
current paper. The boxes are color coded as follows: continuity results in
red boxes require additional geometric constraints and the strongest notion
of admissibility of inhomogeneous data, and were proved in \cite{KoRiSaSh2}%
\footnote{%
The arXiv article \cite{KoRiSaSh2} contains all of the continuity results
stated here for inhomogeneous equations, that require the strong constraint $%
F_{3,\sigma }$ on the geometry, and also a stronger restriction on the
inhomogeneous data $\phi $. These results were obtained there using an
infinitely degenerate Moser scheme that is considerably more complicated
than the adaptation of the DeGeorgi / Caffarelli / Vasseur scheme used in 
\cite{KoRiSaSh}.}; local boundedness results in red boxes require less
stringent additional geometric constraints and the weakest notion of
admissibility of inhomogeneous data, and were proved in \cite{KoRiSaSh}:
results in black boxes are valid for all structured geometries and require
the intermediate notion of admissibility of inhomogeneous data, and are
proved here; and the single continuity result in the purple box requires a
geometric constraint and the intermediate notion of admissibility of
inhomogeneous data, and is also proved here. In all cases the inhomogeneous
data include bounded measurable functions $\phi $. The above results are
described in more detail below, see e.g. Definition \ref{def A admiss max}
for the meaning of $f$-admissible, and Definition \ref{k sigma} for the
meaning of $F_{k,\sigma }$. 
\begin{table}[h]
\centering%
\begin{tabular}{|c|c|c|c|}
\hline
Data & $\phi (x,y)$ & $\phi (x,\cdot )$ $\rho$ cont unif in $x$ & $0$ \\ 
\hline
$\mathcal{A}(x,y,u(x,y))$ & \textcolor{red}{loc bdd $F_{\sigma},\  \sigma<1$}
& \textcolor{red}{loc bdd $F_{\sigma},\  \sigma<1$} & loc bdd all $F$ \\ 
\  & \textcolor{red}{cont $F_{3,\sigma},\  \sigma<1$} & \textcolor{red}{cont
$F_{3,\sigma},\  \sigma<1$} & \  \\ \hline
$A(x)$ & \textcolor{red}{loc bdd $F_{\sigma},\  \sigma<1$} & %
\textcolor{violet}{cont $F_{\sigma},\  \sigma<1$} & cont all $F$ \\ 
\  & \textcolor{red}{cont $F_{3,\sigma},\  \sigma<1$} & \  & \  \\ \hline
\end{tabular}%
\caption{Brief summary of applications}
\end{table}

More specifically, we will consider replacing the Fedi\u{\i} operator $L_{f}$
above with a more general second order divergence form special quasilinear
operator $\mathcal{L}=\nabla ^{\limfunc{tr}}\mathcal{A}\left( x,y,u\right)
\nabla $ in $\mathbb{R}^{2}$ with bounded measurable coefficients, and we
will consider the special quasilinear equations (special because $\mathcal{A}
$ is independent of $\nabla u$) and restricted linear equations (restricted
because $A$ is independent of $y$),%
\begin{eqnarray*}
\text{\textbf{special\ quasilinear}} &\mathbf{:}&\mathbf{\ \ \ \ \ }\mathcal{%
L}u=\nabla ^{\limfunc{tr}}\mathcal{A}\left( x,y,u\right) \nabla u=\phi , \\
\text{\textbf{restricted linear}} &\mathbf{:}&\mathbf{\ \ \ \ \ }Lu=\nabla ^{%
\limfunc{tr}}A\left( x\right) \nabla u=\phi ,
\end{eqnarray*}%
where $\phi $ is $f$-admissible as in Definition \ref{def A admiss max}.
Roughly speaking, we prove the following five new results for such second
order divergence form operators in the plane with `structured geometry' $f$,
i.e. $F=\ln \frac{1}{f}$ satisfies Definition \ref{def A admiss max} below,
which essentially says that $F$ is a doubling function with some normalizing
conditions.

\begin{notation}
We will often make mention of the plane `geometry' associated with the
functions $f\left( r\right) $ or $F\left( r\right) =\ln \frac{1}{f\left(
r\right) }$. By this we mean the geometry of metric balls defined in \cite[%
Chapter 7]{KoRiSaSh} using the degenerate Riemannian metric $dt^{2}=dx^{2}+%
\frac{1}{f\left( x\right) ^{2}}dy^{2}$, with its associated control distance 
$d_{f}$\footnote{%
Recall that a vector $\mathbf{v}$ is subunit for an invertible symmetric
matrix $A$, i.e. $\left( \mathbf{v}\cdot \mathbf{\xi }\right) ^{2}\leq 
\mathbf{\xi }^{\limfunc{tr}}A\mathbf{\xi }$ for all $\mathbf{\xi }$, if and
only if $\mathbf{v}^{\limfunc{tr}}A\mathbf{v}\leq 1$, see e.g. \cite[Chapter
7]{KoRiSaSh}.}. Given two structured geometries represented by $F\left(
r\right) $ and $G\left( r\right) $, we say that $F$ is \emph{stronger}, or 
\emph{less degenerate}, than $G$, if $F\left( r\right) \leq G\left( r\right) 
$ for sufficiently small $r>0$. More generally, if the $2\times 2$ matrix $%
\mathcal{A}\left( x,y,u\left( x,y\right) \right) $ associated with a
divergence form operator $\mathcal{L}$ is comparable to a diagonal matrix $%
D_{f}=\left[ 
\begin{array}{cc}
1 & 0 \\ 
0 & f\left( x\right) ^{2}%
\end{array}%
\right] $, then we refer to this geometry as being associated with $\mathcal{%
L}$ or with $\mathcal{A}\left( x,y,u\left( x,y\right) \right) $.
\end{notation}

\begin{enumerate}
\item (\textbf{homogeneous maximum principle}) For \emph{any} structured
geometry, a maximum principle holds for weak subsolutions to homogeneous
equations $\mathcal{L}u=0$. This has an extension to all dimensions $n\geq 2$%
.

\item (\textbf{special homogeneous quasilinear equation}) For \emph{any}
structured geometry, weak solutions $u$ to a homogeneous special quasilinear
equation $\mathcal{L}u=0$ are locally bounded.

\item (\textbf{special inhomogeneous quasilinear equation}) If the geometry $%
F$ is stronger than $F_{\sigma }$ for some $\sigma <1$, i.e. $F\left(
r\right) \leq F_{\sigma }\left( r\right) $ for $r>0$ sufficiently small,
then weak solutions $u$ to a special quasilinear equation $\mathcal{L}u=\phi 
$ are locally bounded provided the forcing function $\phi $ is $f$%
-admissible.

\item (\textbf{restricted linear homogeneous equation}) For \emph{any}
structured geometry, if $\mathcal{A}\left( x,y,u\right) =A\left( x\right) $
depends only on $x$, then weak solutions $u$ to the homogeneous restricted
linear equation $Lu=0$ are continuous.

\item (\textbf{restricted linear equation with forcing function continuous
in }$y$) If the geometry $F$ is stronger than $F_{\sigma }$ for some $\sigma
<1$, then weak solutions $u$ to a restricted linear equation $Lu=\phi $ are
continuous provided the $f$-admissible forcing function $\phi \left(
x,y\right) $ is continuous in $y$ with modulus of continuity uniform in $x$.
\end{enumerate}

Statement (1) is Theorem \ref{max copy(1)}, and the reader should have no
difficulty in deriving statements (2), (3), (4) and (5) from the Trace
Method Theorem \ref{trace method theorem} and the maximum principles in
Theorems \ref{geom max princ} and \ref{max copy(1)}. Indeed, statements (2)
and (4) for homogeneous equations use Theorems 2 and 3; while statements (3)
and (5) use Theorems 1 and 3.

The two main new results listed above are statements (2) and (4), which
require no additional geometric assumptions on the geometry of the operator $%
\mathcal{L}$ other than that it is a structured geometry\footnote{%
These conclusions were obtained in \cite{KoRiSaSh} under stronger geometric
assumptions on the operator $\mathcal{L}$, namely $F\geq F_{\sigma }$ for
local boundedness, and $F\geq F_{3,\sigma }$ for continuity, where $0<\sigma
<1$.}; thus giving results closer in spirit to Fedi\u{\i}'s theorem, which
required no geometric assumptions other than that $f$ is positive away from $%
0$. After a section on preliminaries, which makes precise the conditions
surrounding our equations, the following two sections prove the new
homogeneous maximum principle in $\mathbb{R}^{n}$, and the Trace Method
Theorem in $\mathbb{R}^{2}$ respectively.

\subsection{Preliminaries}

We begin with the second order special quasilinear equation (where only $u$,
and not $\nabla u$, appears nonlinearly),%
\begin{equation}
\mathcal{L}u\equiv \nabla ^{\limfunc{tr}}\mathcal{A}\left( x,y,u\left(
x,y\right) \right) \nabla u=\phi ,\ \ \ \ \ \left( x,y\right) \in \Omega ,
\label{eq_0}
\end{equation}%
where $\Omega $ is a bounded domain in the plane $\mathbb{R}^{2}$, and we
assume the following quadratic form condition on the `quasilinear' matrix $%
\mathcal{A}(x,y,z)$, 
\begin{equation}
c\,\xi ^{T}D_{f}(x)\xi \leq \xi ^{T}\mathcal{A}(x,y,z)\xi \leq C\,\xi
^{T}D_{f}(x)\xi \ ,  \label{struc_0}
\end{equation}%
for a.e. $\left( x,y\right) \in \Omega $ and all $z\in \mathbb{R}$, $\xi \in 
\mathbb{R}^{2}$, where $c,C$ are positive constants. Equivalently, the $%
2\times 2$ matrix $\mathcal{A}\left( x,y,z\right) $ has bounded measurable
coefficients and is comparable to the following diagonal matrix $D_{f}\left(
x\right) $ depending only on $x$,%
\begin{equation*}
D_{f}\left( x\right) \equiv \left[ 
\begin{array}{cc}
1 & 0 \\ 
0 & f\left( x\right) ^{2}%
\end{array}%
\right] .
\end{equation*}

Define the $f$-gradient by%
\begin{equation}
\nabla _{f}=D_{f}\left( x\right) \nabla \ ,  \label{def A grad}
\end{equation}%
and the associated degenerate Sobolev space $W_{f}^{1,2}\left( \Omega
\right) $ to have norm%
\begin{equation*}
\left\Vert v\right\Vert _{W_{f}^{1,2}}\equiv \sqrt{\int_{\Omega }\left(
\left\vert v\right\vert ^{2}+\nabla v^{\func{tr}}D_{f}\nabla v\right) }=%
\sqrt{\int_{\Omega }\left( \left\vert v\right\vert ^{2}+\left\vert \nabla
_{f}v\right\vert ^{2}\right) }.
\end{equation*}%
Note that if $A\left( x\right) $ is comparable to $D_{f}\left( x\right) $,
then%
\begin{equation*}
\left\Vert v\right\Vert _{W_{f}^{1,2}}\approx \sqrt{\int_{\Omega }\left(
\left\vert v\right\vert ^{2}+\nabla v^{\func{tr}}A\nabla v\right) },
\end{equation*}%
which shows that $\nabla _{f}$ is an appropriate gradient to use in
connection with the operator $\mathcal{L}$. We say $u\in W_{f}^{1,2}\left(
\Omega \right) $ is a $W_{f}^{1,2}\left( \Omega \right) $-weak solution to $%
\mathcal{L}u=\phi $ if%
\begin{equation*}
-\int \left( \nabla w\right) ^{\limfunc{tr}}\mathcal{A}\left( x,u\left(
x\right) \right) \nabla u=\int \phi w,\ \ \ \ \ \text{for all }w\in
W_{f}^{1,2}\left( \Omega \right) _{0}\ .
\end{equation*}

We will assume that the degeneracy function $f\left( r\right) =e^{-F\left(
r\right) }$ is even, and\ that there is $R>0$ such that $F$ satisfies the
following five structure conditions from \cite{KoRiSaSh} for some constants $%
C\geq 1$ and $\varepsilon >0$.

\begin{definition}[structure conditions]
\label{structure conditions}A twice continuously differentiable function $%
F:\left( 0,R\right) \rightarrow \mathbb{R}$ is said to satisfy\emph{\ }%
geometric structure conditions, or to be a structured geometry, if:

\begin{enumerate}
\item $\lim_{x\rightarrow 0^{+}}F\left( x\right) =+\infty $;

\item $F^{\prime }\left( x\right) <0$ and $F^{\prime \prime }\left( x\right)
>0$ for all $x\in (0,R)$;

\item $\frac{1}{C}\left\vert F^{\prime }\left( r\right) \right\vert \leq
\left\vert F^{\prime }\left( x\right) \right\vert \leq C\left\vert F^{\prime
}\left( r\right) \right\vert $ for $\frac{1}{2}r<x<2r<R$;

\item $\frac{1}{-xF^{\prime }\left( x\right) }$ is increasing in the
interval $\left( 0,R\right) $ and satisfies $\frac{1}{-xF^{\prime }\left(
x\right) }\leq \frac{1}{\varepsilon }\,$for $x\in (0,R)$;

\item $\frac{F^{\prime \prime }\left( x\right) }{-F^{\prime }\left( x\right) 
}\approx \frac{1}{x}$ for $x\in (0,R)$.
\end{enumerate}
\end{definition}

\begin{definition}
\label{k sigma}For $0<r<\infty $ define%
\begin{eqnarray*}
F_{\sigma }\left( r\right) &\equiv &\left( \frac{1}{r}\right) ^{\sigma },\ \
\ \ \ 0<\sigma <1, \\
F_{k,\sigma }\left( r\right) &\equiv &\left( \ln \frac{1}{r}\right) \left(
\ln ^{\left( k\right) }\frac{1}{r}\right) ^{\sigma },\ \ \ \ \ 0<\sigma <1%
\text{ and }k\in \mathbb{N}.
\end{eqnarray*}
\end{definition}

The functions $F_{\sigma }$ and $F_{k,\sigma }$ are examples of functions
satisfying geometric structure conditions as above. Note that $f_{\sigma
}=e^{-F_{\sigma }}$ vanishes to infinite order at $r=0$, and that $f_{\sigma
}$ vanishes to a faster order than $f_{\sigma ^{\prime }}$ if $\sigma
>\sigma ^{\prime }$. A similar remark applies to $f_{k,\sigma
}=e^{-F_{k,\sigma }}$. The first part of the next definition originates in 
\cite[see Definition 4]{KoRiSaSh}.

\begin{definition}
\label{def A admiss max}Fix a bounded domain $\Omega \subset \mathbb{R}^{2}$%
. Define the space $X_{f}\left( \Omega \right) $ to consist of all functions 
$\phi $ on $\mathbb{R}^{n}$ such that 
\begin{equation*}
\Vert \phi \Vert _{X_{f}\left( \Omega \right) }\equiv \sup_{v\in \left(
W_{f}^{1,1}\right) _{0}(\Omega )}\frac{\int_{\Omega }\left\vert v\phi
\right\vert \,dy}{\int_{\Omega }\Vert \nabla _{f}v\Vert \,dy}<\infty .
\end{equation*}%
We say that $\phi $ is $f$\emph{-admissible} in $\Omega $ if both $\phi \in
X_{f}\left( \Omega \right) $ and $\phi $ satisfies the following $L^{q}$
growth condition in $\Omega $,%
\begin{equation*}
\left\Vert \phi \right\Vert _{L_{\limfunc{growth}}^{q}\left( \Omega \right)
}\equiv \sup_{\substack{ \left( x,y\right) \in \Omega \setminus \left\{ y%
\text{-axis}\right\}  \\ B\left( \left( x,y\right) ,\frac{\left\vert
x\right\vert }{2}\right) \subset \Omega }}\left\Vert \phi \right\Vert
_{L^{q}\left( B\left( \left( x,y\right) ,\frac{\left\vert x\right\vert }{2}%
\right) \right) }<\infty ,\text{\ \ \ \ for some }q>\frac{n}{2}.
\end{equation*}%
We norm the $f$-admissible functions with%
\begin{equation*}
\left\Vert \phi \right\Vert _{f-\limfunc{adm}\left( \Omega \right) }\equiv
\Vert \phi \Vert _{X_{f}\left( \Omega \right) }+\left\Vert \phi \right\Vert
_{L_{\limfunc{growth}}^{q}\left( \Omega \right) }.
\end{equation*}
\end{definition}

The point of including $L_{\limfunc{growth}}^{q}$ in the definition of $f$%
\emph{-admissible} is so that we can apply standard elliptic theory as in 
\cite{GiTr} away from the $y$-axis with appropriate uniformity. We will
apply the definition of $f$-admissible to forcing functions $\phi $ only for
structured geometries $f$. In connection with the definition of $\Vert \phi
\Vert _{X_{f}\left( \Omega \right) }$, note that $\int_{\Omega }\Vert \nabla
_{f}v\Vert \,dy\approx \Vert v\Vert _{W_{f}^{1,1}\left( \Omega \right) }$ by
the $1-1$ Poincar\'{e} inequality analogous to (\ref{Sob-2-2-geo}) in
Proposition \ref{sob_global-2-2}\ below. Finally note that bounded functions
are $f$-admissible; $\left\Vert \phi \right\Vert _{f-\limfunc{adm}\left(
\Omega \right) }\lesssim \left\Vert \phi \right\Vert _{L^{\infty }\left(
\Omega \right) }$.

\begin{definition}
We say a function $u\in W_{f}^{1,2}\left( \Omega \right) $ is \emph{bounded
by a constant }$\ell \in \mathbb{R}$ on the boundary $\partial \Omega $ if $%
\left( u-\ell \right) ^{+}=\max \left\{ u-\ell ,0\right\} \in \left(
W_{f}^{1,2}\right) _{0}\left( \Omega \right) $. We define $\sup_{x\in
\partial \Omega }u\left( x\right) $ to be $\inf \left\{ \ell \in \mathbb{R}%
:\left( u-\ell \right) ^{+}\in \left( W_{f}^{1,2}\right) _{0}\left( \Omega
\right) \right\} $.
\end{definition}

Before we start stating our main new results, we recall the geometric
maximum principle for weak subsolutions to inhomogeneous equations given in 
\cite{KoRiSaSh} under additional restrictions on the geometry $F$ associated
with the form $A$ of the operator. We restrict our attention to the case $%
n=2 $ as that is the only dimension in which we obtain new results for $%
W_{f}^{1,2}\left( \Omega \right) $-weak solutions. Namely, we assume that $%
f(x)\neq 0$ if $x\neq 0$, and that $F$ satisfies the five geometric
structure conditions in Definition \ref{structure conditions}. Note that the
admissibility requirement for $\phi $ in the next theorem is the weakest
one, $\phi \in X_{f}\left( \Omega \right) $, and which is used in \cite%
{KoRiSaSh}.

\begin{theorem}[inhomogeneous maximum principle]
\label{geom max princ}Suppose that $F=\ln \frac{1}{f}$ is a structured
geometry, i.e. satisfies the structure conditions in Definition \ref%
{structure conditions}, and is stronger than $F_{\sigma }$ for some $%
0<\sigma <1$. Assume that $u$ is a weak subsolution to $\mathcal{L}u=\phi $
in a domain $\Omega \subset \mathbb{R}^{2}$, where $\mathcal{L}$ $=\nabla ^{%
\limfunc{tr}}\mathcal{A}\nabla $ and $\mathcal{A}\approx D_{f}$ has bounded
measurable coefficients, and $\phi \in X_{f}\left( \Omega \right) $.
Moreover, suppose that $u$ is bounded in the weak sense on the boundary $%
\partial \Omega $. Then $u$ is globally bounded in $\Omega $ and satisfies%
\begin{equation*}
\sup_{\Omega }u\leq \sup_{\partial \Omega }u+C\left\Vert \phi \right\Vert
_{X\left( \Omega \right) }\ .
\end{equation*}
\end{theorem}

We will use the above maximum principle when dealing with inhomogeneous
equations in the plane. On the other hand, when dealing with homogeneous
equations, we will use an improved maximum principle, valid for more general
geometries, and which is our first main new theorem.

\subsection{Statement of the two main results}

Since our homogeneous maximum principle holds in higher dimensions as well,
we will give the statement and proof for domains $\Omega \subset \mathbb{R}%
^{n}$. We refer to \cite{KoRiSaSh} for the straightforward extension of the
planar definitions used here to higher dimensions, noting in particular that 
$D_{f}\left( x_{1},...x_{n}\right) $ is the $n\times n$ diagonal matrix with
diagonal entries $\left\{ 1,...1,f\left( x_{1}\right) ^{2}\right\} $. Our
first main theorem is a maximum principle for homogeneous equations in $%
\mathbb{R}^{n}$ that holds for all structured geometries.

\begin{theorem}[homogeneous maximum principle]
\label{max copy(1)}Suppose that $F=\ln \frac{1}{f}$ is a structured
geometry, i.e. satisfies the geometric structure conditions in Definition %
\ref{structure conditions}. Assume that $u$ is a weak subsolution to $%
\mathcal{L}u=0$ in a domain $\Omega \subset \mathbb{R}^{n}$ with $n\geq 2$,
where $\mathcal{L}=\nabla ^{\limfunc{tr}}\mathcal{A}\nabla $ and $\mathcal{A}%
\approx D_{f}$ has bounded measurable coefficients. Moreover, suppose that $%
u $ is bounded in the weak sense on the boundary $\partial \Omega $. Then 
\begin{equation*}
\sup_{\Omega }u\leq \sup_{\partial \Omega }u\ .
\end{equation*}
\end{theorem}

Our second main theorem yields a new method for obtaining local boundedness
and continuity of weak solutions in the plane, given that we already have an
`appropriate' maximum principle. In order to combine statements using either
the homogeneous or inhomogeneous maximum principles, it is convenient to
define precisely what we mean by an `appropriate' maximum principle in the
plane.

\begin{definition}
Let $\phi \in L_{\limfunc{loc}}^{2}\left( \Omega \right) $ and $\mathcal{L}%
=\nabla ^{\limfunc{tr}}\mathcal{A}\nabla $ where $\mathcal{A}\approx D_{f}$
has bounded measurable coefficients and a structured geometry in $\Omega $.\
An equation $\mathcal{L}v=\phi $ satisfies the \emph{Maximum Principle
Property in }$\Omega $, or $\mathcal{MPP}$ for short, if for every open
rectangle $R=\left( a,b\right) \times \left( c,d\right) $ with closure
contained in $\Omega $, there is a constant $C_{R}$ such that 
\begin{equation*}
\sup_{\Omega }v\leq \sup_{\partial \Omega }v+C_{R}\ ,
\end{equation*}%
for all weak solutions $v$ of $\mathcal{L}v=\phi $ that are bounded in the
weak sense on $\partial \Omega $.
\end{definition}

The new method, which we refer to as the trace method, is embodied in the
next theorem, for which we need the following somewhat technical definition
of a modulus of continuity associated with a structured geometry.

\begin{definition}
For any structured geometry $F=\ln \frac{1}{f}$, let $\omega _{f}$ be the
modulus of continuity associated to $f$ as defined in (\ref{def Gamma}) and (%
\ref{def omega}) below.
\end{definition}

For any modulus of continuity $\omega $, define a difference operator in the
second variable by $D_{\mathbf{e}_{2},\delta }^{\omega }h\left( x,y\right) =%
\frac{h\left( x,y+\delta \right) -h\left( x,y\right) }{\omega \left( \delta
\right) }$.

\begin{theorem}[trace method]
\label{trace method theorem}Let $\phi $ satisfy the $L^{q}$ growth condition 
$\left\Vert \phi \right\Vert _{L_{\limfunc{growth}}^{q}\left( \Omega \right)
}<\infty $ for the operator $\mathcal{L}$ in the domain $\Omega \subset 
\mathbb{R}^{2}$, for some $q>\frac{n}{2}$.\ Suppose the equation $\mathcal{L}%
v=\nabla \mathcal{A}\nabla v=\phi $ satisfies the $\mathcal{MPP}$, i.e. the
Maximum Principle Property, in $\Omega $ with $\mathcal{A}$ as above,\ and
suppose that $u\in W_{f}^{1,2}\left( \Omega \right) $ is a weak solution to
this equation, i.e. $\mathcal{L}u=\phi $ in $\Omega $. Then

\begin{enumerate}
\item $u$ is locally bounded;

\item if in addition $\phi \left( x,y\right) $ has modulus of continuity $%
\rho $ in the $y$-variable uniformly in $x$, and if the equation $\mathcal{L}%
v=D_{\mathbf{e}_{2},\delta }^{\omega }\phi $ with $\omega \equiv \max
\left\{ \omega _{f},\rho \right\} $ satisfies the $\mathcal{MPP}$, and if $%
\mathcal{A}\left( x,y,u\right) =A\left( x\right) $ is independent of $y$ and 
$u$, then $u\in Lip_{\omega }\left( \Omega ^{\prime }\right) $ for all $%
\Omega ^{\prime }\Subset \Omega $.
\end{enumerate}
\end{theorem}

\section{Proof of the homogeneous maximum principle in $\mathbb{R}^{n}$}

In this section we prove the new homogeneous maximum principle Theorem \ref%
{max copy(1)}, valid in all dimensions. We start with the Caccioppoli
inequality, generalizing a similar inequality in Caffarelli and Vasseur \cite%
{CaVa}.

\begin{proposition}
\label{new hom max princ}Let $B$ be a ball, and $u\in \left(
W_{f}^{1,2}\right) _{0}(B)$ be a weak subsolution to $\mathcal{L}u=\phi $
with $\mathcal{L}$ as in (\ref{eq_0}), (\ref{struc_0}), and $\phi \in
X_{f}\left( B\right) $. Suppose there is a constant $P>0$, and a nonnegative
function $v\in W_{f}^{1,2}(B)$ such that 
\begin{equation}
\left\Vert \phi \right\Vert _{X_{f}\left( B\right) }^{2}\leq Pv(x),\quad 
\text{a.e.}\ x\in \{u>0\}\cap B.  \label{P_bound'}
\end{equation}%
Then 
\begin{equation}
\int_{B}|\nabla _{f}u_{+}|^{2}d\mu \leq CP^{2}\int_{B}v^{2}d\mu ,
\label{Cacc_max}
\end{equation}%
where $d\mu =\frac{dx}{\left\vert B\right\vert }$, and $C\ $depending only
on the constant $c$ in (\ref{struc_0}). A similar result holds with $u_{-}$
in place of $u_{+}$.
\end{proposition}

\begin{proof}
Since $u$ is a weak subsolution to $\mathcal{L}u=\phi $, we have using as
test function $w\equiv u_{+}\in \left( W_{f}^{1,2}\right) _{0}(B)$, that 
\begin{eqnarray*}
\int \nabla \left( u_{+}\right) \mathcal{A}\nabla u_{+}d\mu  &=&\int \nabla
u_{+}\mathcal{A}\nabla ud\mu \leq -\int u_{+}\phi d\mu  \\
&\leq &\left\Vert \phi \right\Vert _{X_{f}\left( B\right) }\int \left\vert
\nabla _{f}u_{+}\right\vert d\mu \leq CP\int v\left\vert \nabla
_{f}u_{+}\right\vert d\mu \ ,
\end{eqnarray*}%
where for the last inequality we used conditions (\ref{struc_0}) and (\ref%
{P_bound'}). Using H\"{o}lder's inequality and (\ref{struc_0}) this gives 
\begin{equation*}
\int \left\vert \nabla _{f}u_{+}\right\vert ^{2}d\mu \leq \frac{1}{c}\int
\nabla \left( u_{+}\right) \mathcal{A}\nabla u_{+}d\mu \leq CP^{2}\int
v^{2}d\mu ,
\end{equation*}%
which is (\ref{Cacc_max}). Using $w=u_{-}$ as a test function we obtain the
last statement of the Proposition.
\end{proof}

Now we can prove the abstract maximum principle for homogeneous equations,
assuming only the `straight-across' Sobolev inequality 
\begin{equation}
\left\Vert w\right\Vert _{L^{2}\left( \Omega \right) }\leq C\left( \Omega
\right) \left\Vert \nabla _{f}w\right\Vert _{L^{2}\left( \Omega \right) },\
\ \ \ \ w\in \left( W_{f}^{1,2}\right) _{0}\left( \Omega \right) 
\label{Sob_2-2}
\end{equation}

\begin{theorem}
\label{2D max hom} Let $\Omega $ be a bounded open subset of $\mathbb{R}^{n}$%
, and assume the global Sobolev inequality (\ref{Sob_2-2}) holds. Assume
that $u$ is a weak subsolution to the homogeneous equation $\mathcal{L}u=0$
in $\Omega $, and that $u$ is bounded on the boundary $\partial \Omega $.
Then we have 
\begin{equation*}
\sup_{\Omega }u\leq \sup_{\partial \Omega }u.
\end{equation*}
\end{theorem}

\begin{proof}
If $u$ is a weak subsolution to $\mathcal{L}u=0$, then so is $%
u-\sup_{\partial \Omega }u$, therefore we can assume $u\leq 0$ on $\partial
\Omega $. Since $\phi \equiv 0$, we can use (\ref{Cacc_max}) with $v\equiv 0$
and (\ref{Sob_2-2}) applied to $w=u_{+}$ to obtain 
\begin{equation*}
\int u_{+}^{2}\leq 0,
\end{equation*}%
which implies $u\leq 0$ in $\Omega $. Applying this to $u-\sup_{\partial
\Omega }u$ gives the result.
\end{proof}

To obtain the `geometric' version of this theorem, namely Theorem \ref{max
copy(1)}, we need to show (\ref{Sob_2-2}). For this we first recall a
proposition from \cite[Proposition 76]{KoRiSaSh}.

\begin{proposition}
\label{vanishing}Let the balls $B(0,r)$ and the degenerate gradient $\nabla
_{f}$ be as above for a structured geometry. There exists a constant $C$
such that the proportional vanishing $L^{1}$-Sobolev inequality 
\begin{equation}
\int_{B(0,r)}\left\vert w\right\vert dx\leq Cr\int_{B(0,2r)}|\nabla _{f}w|dx,
\label{proportional_sob}
\end{equation}%
holds for any Lipschitz function $w$ that vanishes on a subset $E$ of the
ball $B\left( 0,r\right) $ with $\left\vert E\right\vert \geq \frac{1}{2}%
\left\vert B\left( 0,r\right) \right\vert $, and all sufficiently small $r>0$%
.
\end{proposition}

\begin{proposition}
\label{sob_global-2-2} Suppose that $F$ satisfies the geometric structure
conditions in Definition \ref{structure conditions}. Then the following $%
(2,2)$ global Sobolev inequality holds with geometry $F$ 
\begin{equation}
\left\Vert w\right\Vert _{L^{2}\left( \Omega \right) }\leq C\left( \Omega
\right) \left\Vert \nabla _{f}w\right\Vert _{L^{2}\left( \Omega \right) },\
\ \ \ \ w\in \left( W_{f}^{1,2}\right) _{0}\left( \Omega \right)
\label{Sob-2-2-geo}
\end{equation}
\end{proposition}

\begin{proof}
As in the proof of a corresponding proposition in \cite[Proposition 81]%
{KoRiSaSh}, it suffices by using a partition of unity to suppose that $%
\Omega $ is bounded. Then choose a ball $B(0,r_{0})$ containing $\Omega $
and extend $w$ to be $0$ outside $\Omega $ so that $w\in \left(
W_{f}^{1,2}\right) _{0}\left( B(0,r_{0})\right) $. Next, choose $R>r_{0}$
s.t. $E\equiv B(0,R)\backslash B(0,r_{0})$ satisfies $|E|\geq |B(0,R)|/2$.
Then we can apply Proposition \ref{vanishing} to $w^{2}\in \left(
W_{f}^{1,1}\right) _{0}\left( B(0,R)\right) $ to obtain 
\begin{equation*}
\int_{\Omega }w^{2}dx=\int_{B(0,R)}w^{2}dx\leq CR\int_{B(0,R)}|\nabla
_{f}w^{2}|dx=2CR\int_{\Omega }|w||\nabla _{f}w|dx.
\end{equation*}%
Using H\"{o}lder's inequality we conclude (\ref{Sob-2-2-geo}).
\end{proof}

Thus Proposition \ref{sob_global-2-2}, together with Theorem \ref{2D max hom}%
, prove the geometric maximum principle in Theorem \ref{max copy(1)}.

\section{Proof of the Trace Method Theorem in $\mathbb{R}^{2}$}

We will prove the Trace Method Theorem in eight steps. Conclusion (1) of
Theorem \ref{trace method theorem} will follow from the first three steps,
where the first two will establish `smoothness' properties of functions $%
u\in W_{f}^{1,2}\left( \Omega \right) $, where it is crucial that $\Omega $
is a planar domain, and the third requires that $u$ be a weak solution.
Conclusion (2) will then follow from an additional five steps, two of which
are refinements of Steps two and three. It suffices to consider the case $%
\Omega =\left( -1,1\right) \times \left( -1,1\right) $, which we assume in
all eight steps below. We also use the notation $\Omega _{a,b}^{c,d}\equiv
\left( a,b\right) \times \left( c,d\right) $ for $-1\leq a<b\leq 1$ and $%
-1\leq c<d\leq 1$.

\subsection{Local boundedness of weak solutions}

Here we will prove Conclusion (1) of the Trace Method Theorem. We begin with
Lebesgue's differentiation theorem and maximal function for Hilbert space
valued functions on an interval $\left( c,d\right) $.

\begin{lemma}
\label{Leb diff Hilbert}Suppose that $\mathcal{H}$ is a separable Hilbert
space and that $F\in L_{\mathcal{H}}^{2}\left( \left( c,d\right) \right) $,
i.e. $F:\left( c,d\right) \rightarrow \mathcal{H}$ and 
\begin{equation*}
\left\Vert F\right\Vert _{L_{\mathcal{H}}^{2}\left( \left( c,d\right)
\right) }=\sqrt{\int_{c}^{d}\left\vert F\left( y\right) \right\vert _{%
\mathcal{H}}^{2}dy}<\infty .
\end{equation*}%
Then for almost every $x\in \left( c,d\right) $ we have both%
\begin{eqnarray*}
&&\lim_{I\searrow \left\{ y\right\} }\frac{1}{\left\vert I\right\vert }%
\int_{I}\left\vert F\left( t\right) -F\left( y\right) \right\vert _{\mathcal{%
H}}^{2}dt=0, \\
&&M_{2}F\left( y\right) =\sup_{I:\ y\in I}\sqrt{\frac{1}{\left\vert
I\right\vert }\int_{I}\left\vert F\left( t\right) \right\vert _{\mathcal{H}%
}^{2}dt}<\infty .
\end{eqnarray*}%
Moreover we have the weak type estimate%
\begin{equation*}
\left\vert \left\{ y\in \left( c,d\right) :M_{2}F\left( y\right) >\lambda
\right\} \right\vert \leq \frac{5}{\lambda ^{2}}\left\Vert F\right\Vert _{L_{%
\mathcal{H}}^{2}\left( \left( c,d\right) \right) }^{2}\ ,\ \ \ \ \ \lambda
>0.
\end{equation*}
\end{lemma}

\begin{proof}
This is proved exactly as in the classical case when the Hilbert space $%
\mathcal{H}$ is the scalar field $\mathbb{R}$.
\end{proof}

We now apply Lemma \ref{Leb diff Hilbert} in the case $\mathcal{H}%
=L^{2}\left( \left( a,b\right) \right) $, so that $F\in L_{\mathcal{H}%
}^{2}\left( \left( c,d\right) \right) $ can be realized as a real-valued
function $f\left( x,y\right) $ defined on $\Omega _{a,b}^{c,d}\equiv \left(
a,b\right) \times \left( c,d\right) $ with%
\begin{eqnarray*}
F\left( y\right)  &=&f\left( x,y\right) ,\ \ \ \ \ \text{for }\left(
x,y\right) \in \Omega _{a,b}^{c,d}, \\
\left\Vert F\right\Vert _{L_{\mathcal{H}}^{2}\left( \left[ a,b\right]
\right) } &=&\sqrt{\int_{c}^{d}\left( \int_{a}^{b}\left\vert f\left(
x,y\right) \right\vert ^{2}dx\right) dy}=\left\Vert f\right\Vert
_{L^{2}\left( \Omega _{a,b}^{c,d}\right) }\ .
\end{eqnarray*}%
Conclusion (1) of the Trace Method Theorem can now be completed in three
steps, and Conclusion (2) will require an additional five steps.

\subsubsection{Step one}

Suppose $u\in W_{D}^{1,2}\left( \Omega \right) $, where $D=\left[ 
\begin{array}{cc}
1 & 0 \\ 
0 & 0%
\end{array}%
\right] $, and set%
\begin{equation*}
f_{1}\left( x,y\right) \equiv u\left( x,y\right) \ \text{and }f_{2}\left(
x,y\right) \equiv \frac{\partial u}{\partial x}\left( x,y\right) .
\end{equation*}%
Suppose $-1<a<b<1$ and $z\in \left( -1,1\right) $ satisfies%
\begin{eqnarray}
&&\lim_{I\searrow \left\{ z\right\} }\frac{1}{\left\vert I\right\vert }%
\int_{I}\left( \int_{-1}^{1}\left\vert f_{i}\left( x,y\right) -f_{i}\left(
x,z\right) \right\vert ^{2}dx\right) dy=0,  \label{cond u} \\
&&\lim_{j\rightarrow \infty }\int_{a}^{b}\left\vert \varphi _{\varepsilon
_{j}}\ast f_{i}\left( x,z\right) -f_{i}\left( x,z\right) \right\vert
^{2}dx=0,  \notag \\
&&M_{2}F_{i}\left( z\right) \leq \Gamma \ ,  \notag
\end{eqnarray}%
where $F_{i}\left( y\right) \equiv f_{i}\left( \cdot ,y\right) \in L_{%
\mathcal{H}}^{2}\left( \left( -1,1\right) \right) $ for $i=1,2$, and $%
\left\{ \varphi _{\varepsilon }\right\} _{\varepsilon >0}$ denotes a smooth
approximate identity in the plane. Then for $-1<a<b<1$, we have%
\begin{equation}
\left\Vert u\left( \cdot ,z\right) \right\Vert _{Lip_{\gamma }\left(
a,b\right) }\leq C_{\gamma }\Gamma \ .  \label{u Lip gamma}
\end{equation}

To see this define%
\begin{equation*}
\Phi _{\varepsilon }\left( z\right) \left( x\right) \equiv \varphi
_{\varepsilon }\ast u\left( x,z\right) \text{ and }\widehat{\Phi }%
_{\varepsilon }\left( z\right) \left( x\right) \equiv \varphi _{\varepsilon
}\ast \frac{\partial u}{\partial x}\left( x,z\right) .
\end{equation*}%
Then both $\Phi _{\varepsilon }\left( z\right) \left( x\right) $ and $%
\widehat{\Phi }_{\varepsilon }\left( z\right) \left( x\right) $ are smooth
functions of $x$ satisfying $\Phi _{\varepsilon }\left( z\right) ^{\prime
}\left( x\right) =\widehat{\Phi }_{\varepsilon }\left( z\right) \left(
x\right) $, and so by the Sobolev embedding theorem in dimension one,%
\begin{eqnarray*}
\left\Vert \Phi _{\varepsilon }\left( z\right) \right\Vert _{Lip_{\frac{1}{2}%
}\left( \left( a,b\right) \right) } &\leq &C\left( \left\Vert \Phi
_{\varepsilon }\left( z\right) \right\Vert _{L^{2}\left( \left( a,b\right)
\right) }+\left\Vert \Phi _{\varepsilon }\left( z\right) ^{\prime
}\right\Vert _{L^{2}\left( \left( a,b\right) \right) }\right) \\
&=&C\left( \left\vert \Phi _{\varepsilon }\left( z\right) \right\vert _{%
\mathcal{H}}+\left\vert \widehat{\Phi }_{\varepsilon }\left( z\right)
\right\vert _{\mathcal{H}}\right) \leq CM_{2}\left( \left\vert \Phi
_{\varepsilon }\right\vert +\left\vert \widehat{\Phi }_{\varepsilon
}\right\vert \right) \left( z\right) \ ,
\end{eqnarray*}%
and then by (\ref{cond u}), we have 
\begin{equation*}
\left\Vert \Phi _{\varepsilon }\left( z\right) \right\Vert _{Lip_{\frac{1}{2}%
}\left( \left( a,b\right) \right) }\leq cM_{2}\left( \left\vert \Phi
_{\varepsilon }\right\vert +\left\vert \widehat{\Phi }_{\varepsilon
}\right\vert \right) \left( z\right) \leq c\Gamma \ .
\end{equation*}

Now from the uniform boundedness of the $Lip_{\frac{1}{2}}\left( \left(
a,b\right) \right) $ norms of $\Phi _{\varepsilon }\left( z\right) $, it
follows that for $0<\gamma <\frac{1}{2}$, there is a sequence of functions $%
\Phi _{\varepsilon _{j}}\left( z\right) $ that converges in $Lip_{\gamma
}\left( \left( a,b\right) \right) $ to $V\left( z\right) \in Lip_{\gamma
}\left( \left( a,b\right) \right) $. Thus $\Phi _{\varepsilon _{j}}\left(
z\right) $ also converges in $L^{2}\left( \left( a,b\right) \right) $ to $%
V\left( z\right) $, which by the second line in (\ref{cond u}) coincides
with the function $x\rightarrow u\left( x,z\right) $, i.e. the function $%
U\left( z\right) =u\left( \cdot ,z\right) $. This completes the proof of (%
\ref{u Lip gamma}).

\subsubsection{Step two}

Fix a smooth approximate identity $\left\{ \varphi _{\varepsilon }\right\}
_{\varepsilon >0}$ in the plane and consider the functions $f_{i}\in
L^{2}\left( \Omega \right) $ in Step one. We claim there are points $c\in
\left( -1,-\frac{1}{2}\right) $ and $d\in \left( \frac{1}{2},1\right) $, and
a sequence $\left\{ \varepsilon _{j}\right\} _{j=1}^{\infty }$, such that
for $z\in \left\{ c,d\right\} $,%
\begin{eqnarray}
&&\lim_{I\searrow \left\{ z\right\} }\frac{1}{\left\vert I\right\vert }%
\int_{I}\left( \int_{-1}^{1}\left\vert f_{i}\left( x,y\right) -f_{i}\left(
x,z\right) \right\vert ^{2}dx\right) dy=0,  \label{cond'} \\
&&\lim_{j\rightarrow \infty }\int_{q}^{r}\left\vert \varphi _{\varepsilon
_{j}}\ast f_{i}\left( x,z\right) -f_{i}\left( x,z\right) \right\vert
^{2}dx=0,\ \ \ \ \ -1<q<r<1,  \notag \\
&&M_{2}F_{i}\left( z\right) \leq \sqrt{11}\left\Vert F_{i}\right\Vert _{L_{%
\mathcal{H}}^{2}\left( \left( -1,1\right) \right) }\ .  \notag
\end{eqnarray}

Indeed, the first two lines of (\ref{cond'}) hold almost everywhere; the
first line since the set of Lebesgue points have full measure, and the
second line follows from%
\begin{equation*}
\lim_{\varepsilon \rightarrow 0}\int_{s}^{t}\left\{ \int_{q}^{r}\left\vert
\varphi _{\varepsilon }\ast f_{i}\left( x,z\right) -f_{i}\left( x,z\right)
\right\vert ^{2}dx\right\} dz=0,\ \ \ \ \ -1<q<r<1,-1<s<t<1,
\end{equation*}%
since the square root of the integral in braces is a function of $z$ that
converges to $0$ in $L^{2}\left( \left( s,t\right) \right) $, and hence has
an almost everywhere pointwise convergent to $0$ sequence $\left\{
\varepsilon _{j}\right\} _{j=1}^{\infty }$. The third line follows from the
weak type estimate for the maximal operator $M_{2}$, 
\begin{equation*}
\left\vert \left\{ y\in \left( -1,1\right) :M_{2}F_{i}\left( y\right)
>\lambda \right\} \right\vert \leq \frac{5}{\lambda ^{2}}\left\Vert
F_{i}\right\Vert _{L_{\mathcal{H}}^{2}\left( \left( -1,1\right) \right)
}^{2}\ ,
\end{equation*}%
since then 
\begin{equation*}
\left\vert \left\{ y\in \left( -1,1\right) :M_{2}F_{i}\left( y\right) >\sqrt{%
11}\left\Vert F_{i}\right\Vert _{L_{\mathcal{H}}^{2}\left( \left(
-1,1\right) \right) }\right\} \right\vert <\frac{1}{2},
\end{equation*}%
which shows there exist points $c\in \left( -1,-\frac{1}{2}\right) $ and $%
d\in \left( \frac{1}{2},1\right) $ satisfying (\ref{cond'}).

Before proceeding to Step three, we give a lemma which will play a crucial
role in both Steps three and seven.

\begin{lemma}
\label{three and seven}Given $v\in W_{f}^{1,2}\left( \Omega \right) $, $%
\Omega _{q,r}^{s,t}\Subset \Omega $, and a smooth approximate identity $%
\left\{ \varphi _{\varepsilon }\right\} _{0<\varepsilon <1}$, we have 
\begin{equation*}
\left( \varphi _{\varepsilon }\ast v\right) _{+}\rightarrow v_{+}
\end{equation*}%
in the norm of $W_{f}^{1,2}\left( \Omega _{q,r}^{s,t}\right) $ as $%
\varepsilon \rightarrow 0$.
\end{lemma}

\begin{proof}
Let $Y$ be a $C^{1}$ vector field on $\Omega $. Since $\Omega
_{q,r}^{s,t}\Subset \Omega $, we have by a result of Friedrichs \cite{Fri},
see also \cite[see (A.1) in the Appendix]{GaNh}, that the commutator $%
\widetilde{\varphi }_{\varepsilon }\equiv \left[ Y,\varphi _{\varepsilon }%
\right] $ is an integral operator from $L^{2}\left( \Omega \right) $ to $%
L^{2}\left( \Omega _{q,r}^{s,t}\right) $ such that $\left\Vert \widetilde{%
\varphi }_{\varepsilon }w\right\Vert _{L^{2}\left( \Omega
_{q,r}^{s,t}\right) }\longrightarrow 0$ as $\varepsilon \longrightarrow 0$
for all $w\in L^{2}\left( \Omega \right) $. Using this with $Y$ equal to
each of the vector fields $\frac{\partial }{\partial x}$ and $f\left(
x\right) \frac{\partial }{\partial y}$, we then obtain 
\begin{equation}
\varphi _{\varepsilon }\ast v\longrightarrow v\ \text{in }W_{f}^{1,2}\left(
\Omega _{q,r}^{s,t}\right) .  \label{Fried}
\end{equation}

We must show that both%
\begin{eqnarray}
\left( \varphi _{\varepsilon }\ast v\right) _{+} &\rightarrow &v_{+}\text{
in }L^{2}\left( \Omega _{q,r}^{s,t}\right) ,  \label{must show} \\
\nabla _{f}\left[ \left( \varphi _{\varepsilon }\ast v\right) _{+}\right]
&\rightarrow &\nabla _{f}\left[ v_{+}\right] \text{ in }L^{2}\left( \Omega
_{q,r}^{s,t}\right) .  \notag
\end{eqnarray}%
We begin by using the dominated convergence theorem to prove the first line
in (\ref{must show}). Indeed, pick any decreasing sequence $\left\{
\varepsilon _{k}\right\} _{k=1}^{\infty }\subset \left( 0,1\right) $ with $%
\lim_{k\rightarrow \infty }\varepsilon _{k}=0$. From (\ref{Fried}), we see
that there is a subsequence \ converging pointwise almost everywhere, and we
will continue to denote the subsequence by $\left\{ \varepsilon _{k}\right\}
_{k=1}^{\infty }$. Now let $\mathcal{L}\left[ v\right] $ denote the set of
Lebesgue points of $v$, and note that 
\begin{equation*}
\left\{ x\in \Omega _{q,r}^{s,t}:\lim_{k\rightarrow \infty }\varphi
_{\varepsilon _{k}}\ast v\left( x\right) =v\left( x\right) \right\} \subset 
\mathcal{L}\left[ v\right] ,
\end{equation*}%
and of course $\left\vert \Omega _{q,r}^{s,t}\setminus \mathcal{L}\left[ v%
\right] \right\vert =0$. On the set $\mathcal{L}\left[ v\right] $ we have%
\begin{equation*}
\lim_{k\rightarrow \infty }\left[ \varphi _{\varepsilon _{k}}\ast v\left(
x\right) \right] _{+}=\left[ v\left( x\right) \right] _{+}\ ,\ \ \ \ \ x\in 
\mathcal{L}\left[ v\right] ,
\end{equation*}%
and by \cite[Theorem 2 on page 62]{Ste}, the supremum over $k$ satisfies%
\begin{equation*}
\sup_{1\leq k<\infty }\left[ \varphi _{\varepsilon _{k}}\ast v\left(
x\right) \right] _{+}\leq CMv\left( x\right) ,\ \ \ \ \ x\in \Omega
_{q,r}^{s,t}\ ,
\end{equation*}%
where $M$ is the maximal function. Since $Mv\in L^{2}\left( \Omega \right) $
by the maximal theorem, the dominated convergence theorem now yields the
first line in (\ref{must show}).

To prove the second line in (\ref{must show}), we use identities from \cite[%
see (33) on page 1886]{SaW3}, which the reader can easily verify translate
into the following in our notation,%
\begin{equation}
\nabla _{f}\left( v_{+}\right) =\mathbf{1}_{\left\{ v>0\right\} }\nabla _{f}v%
\text{ and }\nabla _{f}\left[ \left( \varphi _{\varepsilon }\ast v\right)
_{+}\right] =\mathbf{1}_{\left\{ \varphi _{\varepsilon }\ast v>0\right\}
}\nabla _{f}\left( \varphi _{\varepsilon }\ast v\right) .  \label{ident}
\end{equation}%
We claim the same argument as above now yields the limit

\begin{equation}
\mathbf{1}_{\left\{ \varphi _{\varepsilon }\ast v>0\right\} }\nabla
_{f}\left( \varphi _{\varepsilon }\ast v\right) \rightarrow \mathbf{1}%
_{\left\{ v>0\right\} }\nabla _{f}\left( v\right) \text{ in }L^{2}\left(
\Omega _{q,r}^{s,t}\right) .  \label{grad limit}
\end{equation}%
Indeed, we see from (\ref{Fried}) that%
\begin{equation*}
\varphi _{\varepsilon }\ast \nabla _{f}v\longrightarrow \nabla _{f}v\ \text{%
in }L^{2}\left( \Omega _{q,r}^{s,t}\right) .
\end{equation*}%
Thus every decreasing sequence in $\left( 0,1\right) $ with limit $0$ at $0$%
, has a subsequence $\left\{ \varepsilon _{k}\right\} _{k=1}^{\infty }$ such
that 
\begin{equation*}
\nabla _{f}\left( \varphi _{\varepsilon _{k}}\ast v\right) =\varphi
_{\varepsilon _{k}}\ast \nabla _{f}v\longrightarrow \nabla _{f}v,\ \ \ \ \ 
\text{pointwise almost everywhere in }\Omega _{q,r}^{s,t}\ .
\end{equation*}%
Then 
\begin{equation*}
\left\{ x\in \Omega _{a,b}^{c,d}:\lim_{k\rightarrow \infty }\varphi
_{\varepsilon _{k}}\ast \nabla _{f}v\left( x\right) =\nabla _{f}v\left(
x\right) \right\} \subset \mathcal{L}\left[ \nabla v\right] ,
\end{equation*}%
where $\left\vert \Omega _{q,r}^{s,t}\setminus \mathcal{L}\left[ v\right]
\right\vert =0$, and it is easily checked that 
\begin{equation*}
\mathbf{1}_{\left\{ \varphi _{\varepsilon }\ast v>0\right\} }\nabla
_{f}\left( \varphi _{\varepsilon }\ast v\right) \rightarrow \mathbf{1}%
_{\left\{ v>0\right\} }\nabla _{f}v,\ \ \ \ \ \text{pointwise almost
everywhere in }\Omega _{q,r}^{s,t}\ .
\end{equation*}%
Thus from (\ref{ident}) we have $\nabla _{f}\left[ \left( \varphi
_{\varepsilon }\ast v\right) _{+}\right] \rightarrow \nabla _{f}\left(
v_{+}\right) $ pointwise almost everywhere in $\Omega _{q,r}^{s,t}$, and
then using $M\nabla _{f}v\in L^{2}\left( \Omega _{q,r}^{s,t}\right) $, the
dominated convergence theorem yields the second line in (\ref{must show}).
\end{proof}

\subsubsection{Step three}

Under the hypotheses of Trace Method Theorem, and with $c,d$ as in Step two,
we claim that%
\begin{equation*}
u\in L^{\infty }\left( \Omega _{-\frac{3}{4},\frac{3}{4}}^{c,d}\right) ,
\end{equation*}%
which then completes the proof of Conclusion (1) of the Trace Method Theorem.

To prove this we begin with the following lemma.

\begin{lemma}
\label{surround copy (1)}Suppose that $u\in W_{f}^{1,2}\left( \left(
-1,1\right) ^{2}\right) \cap C^{\infty }\left( \left( -1,1\right)
^{2}\setminus \text{ y-axis}\right) $ satisfies $\nabla A\nabla u=\phi $,
where $\phi $ is $f$-admissible, and where $A\left( x,y\right) \approx
D_{f}\left( x\right) $. Choose $c,d$ as in (\ref{cond'}) and choose 
\begin{equation*}
a=-\frac{3}{4}\text{ and }b=\frac{3}{4}.
\end{equation*}%
Then if $\Omega _{a,b}^{c,d}\equiv \left( a,b\right) \times \left(
c,d\right) $, we have that $u$ is bounded in the weak sense on $\partial
\Omega _{a,b}^{c,d}$, i.e. there is a constant $\ell \,$such that $\left(
u-\ell \right) ^{+}\in \left( W_{f}^{1,2}\left( \Omega _{a,b}^{c,d}\right)
\right) _{0}$, and in fact one can take 
\begin{equation*}
\ell \equiv C^{\prime }\max \left\{ \left\Vert u\right\Vert
_{W_{f}^{1,2}\left( \Omega \right) },\left\Vert u\right\Vert _{L^{2}\left(
\Omega \right) }+\left\Vert \phi \right\Vert _{L_{\limfunc{growth}%
}^{q}\left( \Omega \right) }\right\} .
\end{equation*}
\end{lemma}

\begin{proof}
By (\ref{u Lip gamma}) and (\ref{cond'}), we have for $z\in \{c,d\}$ that 
\begin{equation*}
\left\Vert \varphi _{\varepsilon _{j}}\ast u\left( x,z\right) \right\Vert
_{L^{\infty }\left( \left( a,b\right) \right) }\leq C^{\prime }\left\Vert
u\right\Vert _{W_{D}^{1,2}\left( \Omega \right) }\leq C^{\prime }\left\Vert
u\right\Vert _{W_{f}^{1,2}\left( \Omega \right) }.
\end{equation*}%
Then from ellipticity away from the $y$-axis, we have for $t\in \{a,b\}$
that 
\begin{equation*}
\left\Vert \varphi _{\varepsilon _{j}}\ast u\left( t,y\right) \right\Vert
_{L^{\infty }\left( \left( c,d\right) \right) }\leq C^{\prime }\left(
\left\Vert u\right\Vert _{L^{2}\left( \Omega \right) }+\left\Vert \phi
\right\Vert _{L_{\limfunc{growth}}^{q}\left( \Omega \right) }\right) .
\end{equation*}%
Define $\ell \equiv C^{\prime }\max \left\{ \left\Vert u\right\Vert
_{W_{f}^{1,2}\left( \Omega \right) },\left\Vert u\right\Vert _{L^{2}\left(
\Omega \right) }+\left\Vert \phi \right\Vert _{L_{\limfunc{growth}%
}^{q}\left( \Omega \right) }\right\} $. Since $\varphi _{\varepsilon
_{j}}\ast u\left( x,y\right) $ is a smooth function in $\Omega $ provided $%
2\varepsilon <\min \{a-1,1-b,c-1,1-d\}$, the above inequalities imply 
\begin{equation*}
\left\vert \varphi _{\varepsilon _{j}}\ast u\left( x,y\right) \right\vert
\mid _{\partial \Omega _{a,b}^{c,d}}\leq \frac{1}{2}\ell \ .
\end{equation*}%
This gives 
\begin{equation*}
\left( \varphi _{\varepsilon _{j}}\ast u\left( x,y\right) -\frac{1}{2}\ell
\right) _{+}=0\quad \text{on}\quad \partial \Omega _{a,b}^{c,d},
\end{equation*}%
and by continuity, 
\begin{equation*}
\limfunc{supp}\left( \varphi _{\varepsilon _{j}}\ast u\left( x,y\right)
-\ell \right) _{+}\Subset \Omega _{a,b}^{c,d}.
\end{equation*}%
Thus we have%
\begin{equation*}
\left( \varphi _{\varepsilon _{j}}\ast u\left( x,y\right) -\ell \right)
_{+}\in \left( W_{f}^{1,2}\right) _{0}\left( \Omega _{a,b}^{c,d}\right) ,
\end{equation*}%
and it remains to show that 
\begin{equation}
\left( \varphi _{\varepsilon _{j}}\ast u-\ell \right) _{+}\rightarrow
(u-\ell )_{+}  \label{tends'}
\end{equation}%
in the norm of $W_{f}^{1,2}\left( \Omega \right) $ as $\varepsilon
\rightarrow 0$. Indeed, since $\left( W_{f}^{1,2}\left( \Omega \right)
\right) _{0}$ is closed in $W_{f}^{1,2}\left( \Omega \right) $, we would
then conclude that $\left( u-l\right) ^{+}\in \left( W_{f}^{1,2}\left(
\Omega \right) \right) _{0}$ as required. To prove (\ref{tends'}), we note
that since $\varphi _{\varepsilon _{j}}\ast u\left( x,y\right) -\ell
=\varphi _{\varepsilon _{j}}\ast \left( u\left( x,y\right) -\ell \right) $,
we may assume without loss of generality that $\ell =0$, and then Lemma \ref%
{three and seven} applies to finish the proof of Lemma \ref{surround copy
(1)}.
\end{proof}

We now use Lemma \ref{surround copy (1)}, together with Steps one and two,
to obtain that $u$ is bounded in the weak sense on $\partial \Omega
_{a,b}^{c,d}$. Finally then, we apply the assumed $\mathcal{MPP}$ for the
equation $\mathcal{L}u=\phi $ to conclude that $u$ is bounded in $\Omega _{-%
\frac{3}{4},\frac{3}{4}}^{c,d}$,%
\begin{equation*}
\sup_{\Omega _{-\frac{3}{4},\frac{3}{4}}^{c,d}}u\leq C^{\prime }\left(
\left\Vert u\right\Vert _{W_{f}^{1,2}\left( \Omega \right) }+\left\Vert \phi
\right\Vert _{L_{\limfunc{growth}}^{q}\left( \Omega \right) }+C_{R}\right) .
\end{equation*}

\subsection{Continuity of weak solutions}

Here we will prove Conclusion (2) of the Trace Method Theorem in five
additional steps. Steps five and seven are refinements of Steps two and
three respectively.

\subsubsection{Step four}

Here we assume that $\phi $ is $f$-admissible in $\Omega $, and in addition
satisfies the following property:%
\begin{equation}
\phi \left( x,y\right) \text{ has modulus of continuity }\rho \text{ in }y%
\text{, uniformly in }x.  \label{phi prop}
\end{equation}%
Let $\omega \left( \delta \right) $ be a modulus of continuity with $\omega
\geq \rho $,\ and set $v_{\delta }\left( x,y\right) \equiv u\left(
x,y+\delta \right) $. Consider the difference operators 
\begin{equation*}
D_{\mathbf{e}_{2},\delta }^{\omega }u\left( x\right) \equiv \frac{u\left(
x,y+\delta \right) -u\left( x,y\right) }{\omega \left( \delta \right) }=%
\frac{v_{\delta }\left( x,y\right) -u\left( x,y\right) }{\omega \left(
\delta \right) }
\end{equation*}%
in the $y$-variable. We claim that if $u$ is a weak solution to $Lu=\phi $,
then 
\begin{equation}
w=D_{\mathbf{e}_{2},\delta }^{\omega }u\text{ is a weak solution to }\left\{ 
\begin{array}{ccc}
Lw=0\text{ } & \text{ if } & \phi \left( x,y\right) \text{ is independent of 
}y \\ 
Lw=\eta \in L^{\infty } & \text{ if } & \phi \left( x,\cdot \right) \in
Lip_{\rho }\text{ uniformly in }x%
\end{array}%
\right. .  \label{Du weak u}
\end{equation}

To see this, note that%
\begin{eqnarray*}
L\left( D_{\mathbf{e}_{2},\delta }^{\omega }u\right) \left( x\right)  &=&%
\frac{Lv_{\delta }\left( x,y\right) -Lu\left( x,y\right) }{\omega \left(
\delta \right) } \\
&=&\frac{\left[ \nabla A\left( x,y\right) \nabla \right] u\left( x,y+\delta
\right) -Lu\left( x,y\right) }{\omega \left( \delta \right) } \\
&=&\frac{\nabla A\left( x,y+\delta \right) \nabla u\left( x,y+\delta \right)
-Lu\left( x,y\right) }{\omega \left( \delta \right) } \\
&=&\frac{Lu\left( x,y+\delta \right) -Lu\left( x,y\right) }{\omega \left(
\delta \right) }=\frac{\phi \left( x,y+\delta \right) -\phi \left(
x,y\right) }{\omega \left( \delta \right) },
\end{eqnarray*}%
since $A\left( x\right) $ is independent of $y$. \ Then if $\phi \left(
x,y\right) $ is independent of $y$, we have $L\left( D_{\mathbf{e}%
_{2},\delta }^{\omega }u\right) \equiv 0$, while if $\phi \left( x,\cdot
\right) \in Lip_{\rho }$ uniformly in $x$, then $L\left( D_{\mathbf{e}%
_{2},\delta }^{\omega }u\right) $ is bounded.

\subsubsection{Step five (a refinement of Step two)}

Fix $-1<a<-\frac{3}{4}$ and $\frac{3}{4}<b<1$ . We claim that for every $%
0<\gamma <\frac{1}{2}$, there is a positive constant $C_{\gamma }$ with the
property that for every $0<\delta <\frac{1}{10}$, there is a set $\Theta
_{\delta }\equiv \left\{ c_{\delta },c_{\delta }+\delta ,d_{\delta
},d_{\delta }+\delta \right\} $ of four points with $-1+\delta <c_{\delta }<-%
\frac{1}{2}$ and $\frac{1}{2}<\ d_{\delta }<1-\delta $, such that $\Omega
_{a,b}^{c_{\delta },d_{\delta }}\Subset \Omega $ and%
\begin{equation}
\left\Vert u\left( \cdot ,y\right) \right\Vert _{Lip_{\gamma }\left(
a,b\right) }\leq C_{\gamma }\left\Vert u\right\Vert _{W_{D}^{1,2}\left(
\Omega \right) }\ ,\ \ \ \ \ \text{for all }u\in W_{D}^{1,2}\left( \Omega
\right) \text{ and all }y\in \Theta _{\delta }\ .  \label{U Lip u}
\end{equation}

To prove this, we first note that if $F\in L_{\mathcal{H}}^{2}\left( \left(
-1,1\right) \right) $, we can realize $F$ as a real-valued function $f\left(
x,y\right) $ defined on $\Omega $ with%
\begin{eqnarray*}
F\left( y\right)  &=&f\left( x,y\right) ,\ \ \ \ \ \text{for }\left(
x,y\right) \in \Omega , \\
\left\Vert F\right\Vert _{L_{\mathcal{H}}^{2}\left( \left( -1,1\right)
\right) } &=&\sqrt{\int_{-1}^{1}\left( \int_{-1}^{1}\left\vert f\left(
x,y\right) \right\vert ^{2}dx\right) dy}=\left\Vert f\right\Vert
_{L^{2}\left( \Omega \right) }\ .
\end{eqnarray*}%
Next fix $0<\delta <\frac{1}{10}$, a smooth approximate identity $\left\{
\varphi _{\varepsilon }\right\} _{\varepsilon >0}$ in the plane, and points $%
c\in \left( -1,-\frac{1}{2}\right) $ and $d\in \left( \frac{1}{2},1\right) $%
. We now use Lemma \ref{Leb diff Hilbert} to choose $c_{\delta }\in \left(
c,-\frac{1}{2}\right) $ and $d_{\delta }\in \left( \frac{1}{2},d\right) $
and a sequence $\left\{ \varepsilon _{j}\right\} _{j=1}^{\infty }$ such
that\ for 
\begin{equation*}
z\in \Theta _{\delta }\equiv \left\{ c_{\delta },c_{\delta }+\delta
,d_{\delta },d_{\delta }+\delta \right\} ,
\end{equation*}%
we have%
\begin{eqnarray}
&&\lim_{I\searrow \left\{ z\right\} }\frac{1}{\left\vert I\right\vert }%
\int_{I}\left( \int_{-1}^{1}\left\vert f\left( x,y\right) -f\left(
x,z\right) \right\vert ^{2}dx\right) dy=0,  \label{cond} \\
&&\lim_{j\rightarrow \infty }\int_{a}^{b}\left\vert \varphi _{\varepsilon
_{j}}\ast f\left( x,z\right) -f\left( x,z\right) \right\vert ^{2}dx=0, 
\notag \\
&&M_{2}F\left( z\right) \leq \sqrt{101}\left\Vert F\right\Vert _{L_{\mathcal{%
H}}^{2}\left( \left( -1,1\right) \right) }\ .  \notag
\end{eqnarray}

To this end, we first note that the set $E$ of points $z$ for which the
first two lines of (\ref{cond}) hold is a set of full measure, since the
first line holds for Lebesgue points, and since the second line follows from%
\begin{equation*}
\lim_{\varepsilon \rightarrow 0}\int_{c}^{d}\left\{ \int_{a}^{b}\left\vert
\varphi _{\varepsilon }\ast f\left( x,z\right) -f\left( x,z\right)
\right\vert ^{2}dx\right\} dz=0.
\end{equation*}%
This last assertion holds because the square root of the integral in braces
is a function of $z$ that converges to $0$ in $L^{2}\left( \left( c,d\right)
\right) $, and hence has an almost everywhere pointwise convergent to $0$
sequence $\left\{ \varepsilon _{j}\right\} _{j=1}^{\infty }$. Thus we can
restrict our attention in what follows to points $z\in E$.

Now suppose, in order to derive a contradiction, that the third line in (\ref%
{cond}) fails. Then for each $j\in \mathbb{N}_{\limfunc{odd}}$ with $j\delta
<\frac{1}{10}$, almost every pair of points $\left( c_{\delta
}^{j},d_{\delta }^{j}\right) $ in $E\times E$ satisfying 
\begin{eqnarray*}
-1+j\delta  &\leq &c_{\delta }^{j}<-1+\left( j+1\right) \delta , \\
1-\left( j+1\right) \delta  &\leq &d_{\delta }^{j}<1-j\delta ,
\end{eqnarray*}%
has the property that%
\begin{equation*}
\max \left\{ M_{2}F\left( c_{\delta }^{j}\right) ,M_{2}F\left( c_{\delta
}^{j}+\delta \right) ,M_{2}F\left( d_{\delta }^{j}\right) ,M_{2}F\left(
d_{\delta }^{j}+\delta \right) \right\} >\lambda \equiv \sqrt{101}\left\Vert
F\right\Vert _{L_{\mathcal{H}}^{2}\left( \left( c,d\right) \right) }.
\end{equation*}%
It now follows that for each $j\in \mathbb{N}_{\limfunc{odd}}$ with $j\delta
<\frac{1}{10}$, \textbf{at least one} of the pairwise disjoint sets 
\begin{eqnarray*}
&&\left\{ M_{2}F>\lambda \right\} \cap \left[ -1+j\delta ,-1+\left(
j+1\right) \delta \right) , \\
&&\left\{ M_{2}F>\lambda \right\} \cap \left[ -1+\left( j+1\right) \delta
,-1+\left( j+2\right) \delta \right) , \\
&&\left\{ M_{2}F>\lambda \right\} \cap \left[ 1-\left( j+1\right) \delta
,1-j\delta \right) , \\
&&\left\{ M_{2}F>\lambda \right\} \cap \left[ 1-\left( j+2\right) \delta
,1-\left( j+1\right) \delta \right) ,
\end{eqnarray*}%
has measure at least $\frac{1}{4}\delta $. Thus we have 
\begin{equation*}
\left\vert \left\{ M_{2}F>\lambda \right\} \right\vert \geq \sum_{j\in 
\mathbb{N}_{\limfunc{odd}};j\delta <\frac{1}{10}}\frac{1}{4}\delta \geq 
\frac{1}{20},
\end{equation*}%
by the pairwise disjointedness of these collections of sets in $j$, and
together with the weak type estimate for $M_{2}$ in Lemma \ref{Leb diff
Hilbert}, we obtain 
\begin{equation*}
\frac{1}{20}\leq \left\vert \left\{ y\in \left( c,d\right) :M_{2}F\left(
y\right) >\lambda \right\} \right\vert \leq \frac{5}{\lambda ^{2}}\left\Vert
F\right\Vert _{L_{\mathcal{H}}^{2}\left( \left( c,d\right) \right) }^{2}\ ,
\end{equation*}%
which contradicts the choice $\lambda =\sqrt{101}\left\Vert F\right\Vert
_{L_{\mathcal{H}}^{2}\left( \left( c,d\right) \right) }$. This completes the
proof of (\ref{cond}).

We can also adapt the above argument to show that if, instead of a single
function $f$, we have two functions $F_{1},F_{2}\in \mathcal{H}=L^{2}\left(
\left( -1,1\right) \right) $, then we can choose $c_{\delta }\in \left( c,-%
\frac{1}{2}\right) $ and $d_{\delta }\in \left( \frac{1}{2},d\right) $ such
that%
\begin{eqnarray}
&&\lim_{I\searrow \left\{ z\right\} }\frac{1}{\left\vert I\right\vert }%
\int_{I}\left( \int_{-1}^{1}\left\vert f_{i}\left( x,y\right) -f_{i}\left(
x,z\right) \right\vert ^{2}dx\right) dy=0,  \label{cond2} \\
&&\text{and }M_{2}F_{i}\left( z\right) \leq 2\sqrt{401}\left\Vert
F_{i}\right\Vert _{L_{\mathcal{H}}^{2}\left( \left( -1,1\right) \right) },\
\ \ i=1,2,  \notag \\
&&\text{for }z=c_{\delta },c_{\delta }+\delta ,d_{\delta },d_{\delta
}+\delta \ .  \notag
\end{eqnarray}%
Indeed, if $M_{2}F_{i}\left( z\right) \leq \sqrt{401}\left\Vert
F_{i}\right\Vert _{L_{\mathcal{H}}^{2}\left( \left( -1,1\right) \right) }$
fails for the four choices of $z$ and each $i=1,2$, then there will be \emph{%
eight} pairwise disjoint sets instead of \emph{four} in the above argument,
and we obtain%
\begin{equation*}
\frac{1}{40}\leq \left\vert \left\{ y\in \left( c,d\right) :\max \left\{
M_{2}F_{1}\left( y\right) ,M_{2}F_{i}\left( y\right) \right\} >\lambda
\right\} \right\vert \leq 2\frac{5}{\lambda ^{2}}\left\Vert F_{i}\right\Vert
_{L_{\mathcal{H}}^{2}\left( \left( -1,1\right) \right) }^{2}.
\end{equation*}

Now we use (\ref{cond2}) to apply Step one with 
\begin{equation*}
\Gamma =\sqrt{401}\left( \left\Vert u\right\Vert _{L^{2}\left( \Omega
\right) }+\left\Vert \frac{\partial u}{\partial x}\right\Vert _{L^{2}\left(
\Omega \right) }\right) =\sqrt{401}\left\Vert u\right\Vert
_{W_{D}^{1,2}\left( \Omega \right) },
\end{equation*}%
in order to obtain the uniform boundedness of the $Lip_{\frac{1}{2}}\left(
\left( a,b\right) \right) $ norms of $\Phi _{\varepsilon }\left( z\right) $.
Then it follows that for $0<\gamma <\frac{1}{2}$, there is a sequence of
functions $\Phi _{\varepsilon _{j}}\left( z\right) $ that converges in $%
Lip_{\gamma }\left( \left( a,b\right) \right) $ to $V\left( z\right) \in
Lip_{\gamma }\left( \left( a,b\right) \right) $. Thus $\Phi _{\varepsilon
_{j}}\left( z\right) $ also converges in $L^{2}\left( \left( a,b\right)
\right) $ to $V\left( z\right) $, which by the second line in (\ref{cond})
coincides with the function $x\rightarrow u\left( x,z\right) $, i.e. $%
U\left( z\right) =u\left( \cdot ,z\right) $. This completes the proof of (%
\ref{U Lip u}).

\subsubsection{Step six}

Recall that for a modulus of continuity $\omega $, we defined the
corresponding\ difference operator $D_{\mathbf{e}_{2},\delta }^{\omega }$ in
the direction $\mathbf{e}_{2}$ by 
\begin{equation*}
D_{\mathbf{e}_{2},\delta }^{\omega }u\left( x\right) \equiv \frac{u\left(
x,y+\delta \right) -u\left( x,y\right) }{\omega \left( \delta \right) }.
\end{equation*}%
Then given $\mathcal{L}$ as above and $-1<a<-\frac{3}{4}$, $\frac{3}{4}<b<1$%
, we claim there is a modulus of continuity $\omega _{f}\left( \delta
\right) $ defined for $0<\delta <\frac{1}{10}$ (see (\ref{def omega}) below
for an explicit formula) such that%
\begin{eqnarray}
&&\left\Vert D_{\mathbf{e}_{2},\delta }^{\omega }u\left( \cdot ,z\right)
\right\Vert _{L^{\infty }\left( a,b\right) }\leq C_{0}=C_{0}\left(
\left\Vert u\right\Vert _{W_{f}^{1,2}\left( \Omega \right) },\left\Vert \phi
\right\Vert _{L_{\limfunc{growth}}^{q}\left( \Omega \right) },C_{R}\right) ,
\label{omega weak u} \\
&&\text{ \ \ \ \ for all weak solutions }u\text{ to }Lu=\phi \text{ and
points }z\in \Theta _{\delta }\ ,  \notag
\end{eqnarray}%
where the constant $C_{R}$ is that arising in the $\mathcal{MPP}$.

To prove this we will apply a classical H\"{o}lder estimate for elliptic
equations in small balls arbitrarily close to the singular $y$ -axis, which
we now describe. Without loss of generality we may assume $x\geq 0$, and
then fix $\beta >0$. Consider a Euclidean ball $B_{\limfunc{Euc}}\equiv B_{%
\limfunc{Euc}}\left( \left( x+2\beta ,y\right) ,\beta \right) $ and a
Euclidean ball $RB_{\limfunc{Euc}}\equiv B_{\limfunc{Euc}}\left( \left(
x+2\beta ,y\right) ,R\right) $ concentric with $B_{\limfunc{Euc}}$ and
having radius $0<R\leq \beta $. Let $\phi $ be $f$-admissible. Then in
particular, $\phi \in L_{\limfunc{growth}}^{q}\left( \Omega \right) $ for
some $q>\frac{n}{2}$. Fix such a $q$ and set 
\begin{equation*}
M\equiv \left\Vert \phi \right\Vert _{L_{\limfunc{growth}}^{q}(\Omega )}.
\end{equation*}%
For elliptic operators, we have from Theorem 8.22 in \cite{GiTr} with
notation used there, the following classical H\"{o}lder estimate for weak
solutions, 
\begin{equation}
\limfunc{osc}_{RB_{\limfunc{Euc}}}u\leq \frac{1}{\gamma }\left( \left( \frac{%
R}{\beta }\right) ^{\alpha }\sup_{B_{\limfunc{Euc}}}\left\vert u\right\vert
+kR^{\alpha }\right) ,\quad R<\beta .  \label{osc_est}
\end{equation}%
We now derive bounds for both $\gamma $ and $\alpha $\ from an examination
of the proofs in \cite{GiTr}. Theorems 8.17 and 8.18 in \cite{GiTr}
immediately yield the following Harnack inequality:

\begin{eqnarray*}
\sup_{B_{\limfunc{Euc}}}u &\leq &C_{H}\left( \inf_{B_{\limfunc{Euc}%
}}u+k\beta ^{1-\frac{n}{q}}\right) , \\
\text{with }C_{H} &\equiv &C(n)^{\frac{\Lambda }{\lambda }},\ \ k=\frac{%
C\left\Vert \phi \right\Vert _{L^{q}(B_{\limfunc{Euc}} )}}{\lambda },
\end{eqnarray*}%
and so $k$ can be bounded above by $\frac{CM}{\lambda }$. For homogeneous
equations, where $k=0$, this is recorded as Theorem 8.20 in \cite{GiTr}.\ An
inspection of the inequality at the top of page 202 in \cite{GiTr} now shows
that we can take $\gamma =1-\frac{1}{C_{H}}\geq \frac{1}{2}$ if $C_{H}\geq 2$%
, and%
\begin{equation*}
\alpha \approx -\ln \left( 1-\frac{1}{2C_{H}}\right) ,
\end{equation*}%
since $\alpha =\left( 1-\mu \right) \frac{\log \gamma }{\log \tau }$ in the
notation of \cite{GiTr}.

With the H\"{o}lder estimate (\ref{osc_est}) in hand, we now note that from
Conclusion (1) of the Trace Method Theorem, already proved in Step three, it
follows that for any Euclidean ball $B_{\limfunc{Euc}}$, 
\begin{equation*}
\sup_{B_{\limfunc{Euc}}}u\leq C\left\{ ||u||_{W_{f}^{1,2}(\Omega
)}+\left\Vert \phi \right\Vert _{L_{\limfunc{growth}}^{q}(\Omega
)}+C_{R}\right\} \ ,
\end{equation*}%
with $C$ independent of $\lambda $. Combining this with (\ref{osc_est}) and
the lower estimate for $\gamma $ gives 
\begin{equation*}
\limfunc{osc}_{B_{\limfunc{Euc}}(R)}u\leq C\left( \left( \frac{R}{\beta }%
\right) ^{\alpha }\left( ||u||_{W_{f}^{1,2}(\Omega )}+\left\Vert \phi
\right\Vert _{L_{\limfunc{growth}}^{q}(\Omega )}+C_{R}\right) +kR^{\alpha
}\right) \ ,
\end{equation*}%
with the constant $C$ independent of $\lambda $.

Our operator is elliptic in $B_{\limfunc{Euc}}$ with the ellipticity
constant satisfying 
\begin{equation*}
\lambda \geq f(\beta )^{2}.
\end{equation*}%
Therefore, for $0<\delta \leq \beta $, and using $k\leq \frac{CM}{\lambda }%
\leq \frac{CM}{f(\beta )^{2}}$, we have%
\begin{eqnarray}
&&\left\vert u(x+2\beta ,y+\delta )-u(x+2\beta ,y)\right\vert \leq \limfunc{%
osc}_{B_{\limfunc{Euc}}(x+2\beta ,y)}u  \label{ellipticity-bound} \\
&\leq &C\left( \frac{\delta }{\beta }\right) ^{\alpha (\beta )}\left(
||u||_{W_{f}^{1,2}(\Omega )}+\left\Vert \phi \right\Vert _{L_{\limfunc{growth%
}}^{q}(\Omega )}+C_{R}\right) +C\left\Vert \phi \right\Vert _{L_{\limfunc{%
growth}}^{q}(\Omega )}\frac{\delta ^{\alpha (\beta )}}{f(\beta )^{2}}  \notag
\\
&\leq &C_{\beta }\delta ^{\alpha (\beta )}\left( ||u||_{W_{f}^{1,2}(\Omega
)}+\left\Vert \phi \right\Vert _{L_{\limfunc{growth}}^{q}(\Omega
)}+C_{R}\right) ,  \notag
\end{eqnarray}%
where 
\begin{equation*}
\alpha (\beta )=-C^{\prime }\ln \left( 1-C^{-\frac{1}{f(\beta )^{2}}}\right)
.
\end{equation*}%
We now need to choose $\delta =\delta (\beta )$ such that 
\begin{equation*}
\left( \frac{\delta }{\beta }\right) ^{\alpha (\beta )}\rightarrow 0\quad 
\text{and}\quad \frac{\delta ^{\alpha (\beta )}}{f(\beta )^{2}}\rightarrow
0\quad \text{as}\quad \beta \rightarrow 0.
\end{equation*}

To satisfy the first condition it is sufficient to require 
\begin{align*}
-\ln \left( 1-C^{-\frac{1}{f(\beta )^{2}}}\right) \ln \frac{\delta }{\beta }%
\rightarrow -\infty & \quad \text{as}\quad \beta \rightarrow 0, \\
C^{-\frac{1}{f(\beta )^{2}}}\ln \frac{\delta }{\beta }\rightarrow -\infty &
\quad \text{as}\quad \beta \rightarrow 0, \\
\ln \frac{\delta }{\beta }=-C^{\frac{2}{f(\beta )^{2}}}& ,
\end{align*}%
which holds for $\delta =\Gamma _{f}\left( \beta \right) $ where%
\begin{equation}
\Gamma _{f}\left( \beta \right) \equiv \beta \exp \left( -\exp \left( \frac{%
C^{\prime }}{f(\beta )^{2}}\right) \right) .  \label{def Gamma}
\end{equation}%
Note that as $\beta \rightarrow 0$ we have $\delta \rightarrow 0$ \textbf{\
very quickly}. Thus, as $\delta \rightarrow 0$, $\beta \rightarrow 0$ 
\textbf{very slowly}. We now verify that with $\delta $ as above we also
have 
\begin{equation*}
\frac{\delta ^{\alpha (\beta )}}{f(\beta )^{2}}\rightarrow 0\quad \text{as}%
\quad \beta \rightarrow 0.
\end{equation*}%
Passing to logarithms again we have 
\begin{align*}
\ln \left( \frac{\delta ^{\alpha (\beta )}}{f(\beta )^{2}}\right) =\alpha
(\beta )\ln \delta +2\ln \frac{1}{f(\beta )}& =C^{\prime }\ln \left( 1-C^{-%
\frac{1}{f(\beta )^{2}}}\right) \left( \exp \left( \frac{C^{\prime }}{%
f(\beta )^{2}}\right) +\ln \frac{1}{\beta }\right) +2\ln \frac{1}{f(\beta )}
\\
& \approx -\exp \left( -\frac{C}{f(\beta )^{2}}\right) \left( \exp \left( 
\frac{C^{\prime }}{f(\beta )^{2}}\right) +\ln \frac{1}{\beta }\right) +2\ln 
\frac{1}{f(\beta )}
\end{align*}%
and the expression converges to $-\infty $ as $\beta \rightarrow 0$ provided
we choose $C^{\prime }$ sufficiently large.

We now calculate a modulus of continuity $\omega \left( \delta \right) $
that ensures the function $D_{\mathbf{e}_{2},\delta }^{\omega }u\left(
x,y\right) $ is bounded uniformly for $y\in \left\{ c_{\delta },d_{\delta
}\right\} $. Using (\ref{ellipticity-bound}) on the second term of line 2,
and (\ref{U Lip u}) on the first and third terms of line 2, we have for $%
0<x<x+2\beta <b$ and $0<\delta <\frac{1}{2}\beta $ (and similarly for $x<0$%
), 
\begin{eqnarray*}
&&\left\vert D_{\mathbf{e}_{2},\delta }^{\omega }u\left( x,y\right)
\right\vert =\frac{1}{\omega \left( \delta \right) }\left\vert u\left(
x,y+\delta \right) -u\left( x,y\right) \right\vert \\
&\leq &\frac{\left\vert u\left( x,y+\delta \right) -u\left( x+2\beta
,y+\delta \right) \right\vert }{\omega \left( \delta \right) }+\frac{%
\left\vert u\left( x+2\beta ,y+\delta \right) -u\left( x+2\beta ,y\right)
\right\vert }{\omega \left( \delta \right) }+\frac{\left\vert u\left(
x+2\beta ,y\right) -u\left( x,y\right) \right\vert }{\omega \left( \delta
\right) } \\
&\leq &\frac{C_{0}}{\omega \left( \delta \right) }\left\{ \beta ^{\gamma
}+\left( \frac{\delta }{\beta }\right) ^{\alpha (\beta )}+\frac{\delta
^{\alpha (\beta )}}{f(\beta )^{2}}+\beta ^{\gamma }\right\} \leq C_{0},
\end{eqnarray*}%
where $C_{0}$ depends on $\left\Vert u\right\Vert _{W_{D}^{1,2}\left( \Omega
\right) }$, $\left\Vert \phi \right\Vert _{f-\limfunc{adm}(\Omega )}$ and $%
C_{R}$, provided we choose $\omega \left( \delta \right) =\omega _{f}\left(
\delta \right) $ where%
\begin{equation}
\omega _{f}\left( \delta \right) \equiv \Gamma _{f}^{-1}\left( \delta
\right) ^{\gamma }+\left( \frac{\delta }{\Gamma _{f}^{-1}\left( \delta
\right) }\right) ^{\alpha (\Gamma _{f}^{-1}\left( \delta \right) )}+\frac{%
\delta ^{\alpha (\Gamma _{f}^{-1}\left( \delta \right) )}}{f(\Gamma
_{f}^{-1}\left( \delta \right) )^{2}}.  \label{def omega}
\end{equation}%
Note that the modulus of continuity $\omega _{f}\left( \delta \right) $ is
increasing on $\left( 0,1\right) $ and satisfies $\lim_{\delta \searrow
0}\omega _{f}\left( \delta \right) =0$.

\begin{remark}
The constant $C_{R}$ in the Step six arguments above can be replaced

\begin{enumerate}
\item by $\left\Vert \phi \right\Vert _{X_{f}\left( \Omega \right) }$ if we
are using the maximum principle in Theorem \ref{geom max princ},

\item by $0$ if we are using the homogeneous maximum principle in Theorem %
\ref{max copy(1)}.
\end{enumerate}
\end{remark}

\subsubsection{Step seven (a refinement of Step three)}

We claim that with $\omega \left( \delta \right) \equiv \max \left\{ \omega
_{f}\left( \delta \right) ,\rho \right\} $, where $\rho $ is the modulous of
continuity of $\phi \left( x,\cdot \right) $ as in (\ref{phi prop}), we have%
\begin{equation*}
\left\Vert D_{\mathbf{e}_{2},\delta }^{\omega }u\right\Vert _{L^{\infty
}\left( \Omega _{-\frac{1}{3},\frac{1}{3}}^{-\frac{1}{3},\frac{1}{3}}\right)
}\leq C,\ \ \ \ \ \text{ with a constant }C\text{ independent of }0<\delta <%
\frac{1}{10}.
\end{equation*}

To prove this we will apply the assumed \emph{Maximum Principle Property in }%
$\Omega $ to the function $D_{\mathbf{e}_{2},\delta }^{\omega }u\left(
x,y\right) $, and for this in turn we will need the following lemma.

\begin{lemma}
\label{surround}Suppose that $u\in W_{f}^{1,2}\left( \left( -1,1\right)
^{2}\right) \cap C^{\infty }\left( \left( -1,1\right) ^{2}\setminus \text{ }y%
\text{-axis}\right) $ satisfies $\nabla A\nabla u=\phi $, where $\phi $ is $f
$-admissible and satisfies (\ref{phi prop}), and $A\left( x\right) \approx
D_{f}\left( x\right) $. Let $0<\delta <\frac{1}{10}$, and choose $c_{\delta
},c_{\delta }+\delta ,d_{\delta },d_{\delta }+\delta $ as in (\ref{cond})
and choose $a=-\frac{3}{4}$ and $b=\frac{3}{4}$. Then with $\Omega _{-\frac{3%
}{4},\frac{3}{4}}^{c,d}=\left( -\frac{3}{4},\frac{3}{4}\right) \times \left(
c,d\right) $, we have that $v_{\delta }\equiv D_{\mathbf{e}_{2},\delta
}^{\omega }u$ is bounded in the weak sense on $\partial \Omega
_{a,b}^{c_{\delta },d_{\delta }}$, i.e. there is a constant $\ell \,$such
that $\left( v_{\delta }-\ell \right) ^{+}\in \left( W_{f}^{1,2}\left(
\Omega _{a,b}^{c_{\delta },d_{\delta }}\right) \right) _{0}$.
\end{lemma}

\begin{proof}
We fix $\delta $ and write $v=v_{\delta }$. From (\ref{omega weak u}), we
have for $z\in \{c_{\delta },d_{\delta }\}$ that 
\begin{equation*}
\left\Vert \varphi _{\varepsilon }\ast v\left( \cdot ,z\right) \right\Vert
_{L^{\infty }\left( \left( a,b\right) \right) }\leq C_{0},
\end{equation*}%
and from (\ref{Du weak u}) and ellipticity away from the $y$-axis, we have
for $t\in \{-\frac{3}{4},\frac{3}{4}\}$ that 
\begin{eqnarray*}
\left\Vert \varphi _{\varepsilon }\ast v\left( t,y\right) \right\Vert
_{L^{\infty }\left( \left( c,d\right) \right) } &\leq &C^{\prime }\left(
\left\Vert v\right\Vert _{L^{2}\left( B_{\limfunc{Euc}}\left( \left( \pm 
\frac{3}{4},y\right) ,\frac{1}{10}\right) \right) }+\left\Vert D_{\mathbf{e}%
_{2},\delta }^{\omega }\phi \right\Vert _{L^{q}\left( B_{\limfunc{Euc}%
}\right) }\right)  \\
&\leq &C^{\prime }\left( \left\Vert u\right\Vert _{W_{f}^{1,2}\left( \Omega
\right) }+\left\Vert \phi \right\Vert _{\limfunc{Lip}_{\omega }}+\left\Vert
\phi \right\Vert _{f-\limfunc{adm}\left( \Omega \right) }+C_{R}\right)  \\
&\leq &C^{\prime }\left( \left\Vert u\right\Vert _{W_{f}^{1,2}\left( \Omega
\right) }+\left\Vert \phi \right\Vert _{\limfunc{Lip}_{\omega
}}+C_{R}\right) ,
\end{eqnarray*}%
since (\ref{ellipticity-bound}) shows that 
\begin{eqnarray*}
\left\Vert v\right\Vert _{L^{2}\left( B_{\limfunc{Euc}}\left( \left( \pm 
\frac{3}{4},y\right) ,\frac{1}{10}\right) \right) } &\leq &\left\Vert D_{%
\mathbf{e}_{2}}^{\omega }u\right\Vert _{L^{\infty }\left( B_{\limfunc{Euc}%
}\left( \left( \pm \frac{3}{4},y\right) ,\frac{1}{10}\right) \right) } \\
&\leq &C^{\prime }\frac{\delta ^{\alpha \left( \pm \frac{3}{4}\right) }}{%
\omega \left( \pm \frac{3}{4}\right) }\left( \left\Vert u\right\Vert
_{W_{f}^{1,2}\left( \Omega \right) }+\left\Vert \phi \right\Vert _{f-%
\limfunc{adm}(\Omega )}\right) .
\end{eqnarray*}

Define 
\begin{equation*}
\ell \equiv 2\max \left\{ C_{0},\left\Vert u\right\Vert _{W_{f}^{1,2}\left(
\Omega \right) }+\left\Vert \phi \right\Vert _{\limfunc{Lip}_{\omega
}}\right\} .
\end{equation*}%
Since $\varphi _{\varepsilon }\ast v\left( x,y\right) $ is a smooth function
in $\Omega $ provided $2\varepsilon <\min \{a+1,1-b,c+1,1-d\}$, the above
inequalities imply 
\begin{equation*}
\left\vert \varphi _{\varepsilon }\ast v\left( x,y\right) \right\vert \mid
_{\partial \Omega _{a,b}^{c_{\delta },d_{\delta }}}\leq \frac{1}{2}\ell \ .
\end{equation*}%
This gives 
\begin{equation*}
\left( \varphi _{\varepsilon }\ast v\left( x,y\right) -\frac{1}{2}\ell
\right) _{+}=0\quad \text{on}\quad \partial \Omega _{a,b}^{c_{\delta
},d_{\delta }},
\end{equation*}%
and by continuity, 
\begin{equation*}
\limfunc{supp}\left( \varphi _{\varepsilon }\ast v\left( x,y\right) -\ell
\right) _{+}\Subset \Omega _{a,b}^{c_{\delta },d_{\delta }}.
\end{equation*}%
Thus we have%
\begin{equation*}
\left( \varphi _{\varepsilon }\ast v\left( x,y\right) -\ell \right) _{+}\in
\left( W_{f}^{1,2}\right) _{0}\left( \Omega _{a,b}^{c_{\delta },d_{\delta
}}\right) ,
\end{equation*}%
and it remains to show that 
\begin{equation}
\left( \varphi _{\varepsilon }\ast v-\ell \right) _{+}\rightarrow (v-\ell
)_{+}  \label{tends}
\end{equation}%
in the norm of $W_{f}^{1,2}\left( \Omega \right) $ as $\varepsilon
\rightarrow 0$. Indeed, since $\left( W_{f}^{1,2}\left( \Omega \right)
\right) _{0}$ is closed in $W_{f}^{1,2}\left( \Omega \right) $, we would
then conclude that $u_{\varepsilon }^{+}=\left( v-l\right) ^{+}\in \left(
W_{f}^{1,2}\left( \Omega \right) \right) _{0}$ as required. So it remains to
prove (\ref{tends}), and since $\varphi _{\varepsilon }\ast v\left(
x,y\right) -\ell =\varphi _{\varepsilon }\ast \left( v\left( x,y\right)
-\ell \right) $, we may assume without loss of generality that $\ell =0$.
Lemma \ref{three and seven} now completes the proof of Lemma \ref{surround}.
\end{proof}

With Lemma \ref{surround} in hand, we can now apply the assumed $\mathcal{MPP%
}$ for the equation $\mathcal{L}v=D_{\mathbf{e}_{2},\delta }^{\omega }\phi $
in $\Omega $ to conclude that $D_{\mathbf{e}_{2},\delta }^{\omega }u\in
L^{\infty }\left( \Omega _{-\frac{1}{3},\frac{1}{3}}^{-\frac{1}{3},\frac{1}{3%
}}\right) $ uniformly in $0<\delta <\frac{1}{10}$.

\subsubsection{Step eight}

In order to complete the proof of Conclusion (2) of the Trace Method
Theorem, it remains to show that 
\begin{equation}
u\in Lip_{\omega }\left( \Omega _{-\frac{1}{3},\frac{1}{3}}^{-\frac{1}{3},%
\frac{1}{3}}\right)   \label{remains}
\end{equation}%
for the modulus of continuity $\omega \left( \delta \right) $ in Steps six
and seven.

For this, suppose we are given points $P=\left( x,y\right) $ and $P+\left(
\delta _{1},\delta _{2}\right) =\left( x+\delta _{1},y+\delta _{2}\right) $,
both near the origin, and set 
\begin{equation*}
\delta \equiv \max \left\{ \sqrt{\delta _{1}},\delta _{2}\right\} .
\end{equation*}%
Then choose a `$\delta $-$\limfunc{good}$' point $z$ near $y$, i.e. such that%
\begin{equation}
\left\vert z-y\right\vert \leq \delta \text{ \ \ \ \ and \ \ \ \ }%
\int_{0}^{1}\left( \left\vert u\left( t,z\right) \right\vert ^{2}+\left\vert 
\frac{\partial u}{\partial x}\left( t,z\right) \right\vert ^{2}\right)
dt\leq \frac{C^{2}}{\delta }.  \label{good z}
\end{equation}%
Indeed, this is possible since if we take $\lambda =\frac{C}{\sqrt{\delta }}$
with $C=\frac{1}{\sqrt{10}\left\Vert u\right\Vert _{W_{f}^{1,2}}}$ in the
weak type estimate in Lemma \ref{Leb diff Hilbert}, we obtain 
\begin{eqnarray*}
&&\left\vert \left\{ y\in \left( c,d\right) :M_{2}\sqrt{\left\vert
u\right\vert ^{2}+\left\vert \frac{\partial u}{\partial x}\right\vert ^{2}}%
\left( y\right) >\frac{C}{\sqrt{\delta }}\right\} \right\vert  \\
&\leq &\frac{5\left( \left\Vert u\right\Vert _{L_{\mathcal{H}}^{2}\left(
\left( c,d\right) \right) }^{2}+\left\Vert \frac{\partial u}{\partial x}%
\right\Vert _{L_{\mathcal{H}}^{2}\left( \left( c,d\right) \right)
}^{2}\right) }{C^{2}}\delta =\frac{\delta }{2}\ ,\ \ \ \ \ 0<\delta <\frac{1%
}{10},
\end{eqnarray*}%
and hence conclude that there is $z\in \left( y,y+\delta \right) $ with $%
M_{2}\sqrt{\left\vert u\right\vert ^{2}+\left\vert \frac{\partial u}{%
\partial x}\right\vert ^{2}}\left( z\right) \leq \frac{C}{\sqrt{\delta }}$.

Now we apply (\ref{u Lip gamma}) in Step one with $\frac{1}{4}<\gamma <\frac{%
1}{2}$ to obtain the inequality%
\begin{equation*}
\left\vert u\left( x+\delta _{1},z\right) -u\left( x,z\right) \right\vert
\leq C_{\gamma }\delta _{1}^{\gamma }\frac{C}{\sqrt{\delta }}\leq C_{\gamma
}\delta ^{2\gamma -\frac{1}{2}}.
\end{equation*}%
Altogether then we have using Step seven that%
\begin{eqnarray*}
&&\left\vert u\left( x+\delta _{1},y+\delta _{2}\right) -u\left( x,y\right)
\right\vert  \\
&\leq &\left\vert u\left( x+\delta _{1},y+\delta _{2}\right) -u\left(
x+\delta _{1},z\right) \right\vert +\left\vert u\left( x+\delta
_{1},z\right) -u\left( x,z\right) \right\vert +\left\vert u\left( x,z\right)
-u\left( x,y\right) \right\vert  \\
&\leq &\omega \left( \left\vert y+\delta _{2}-z\right\vert \right)
+C_{\gamma }\delta ^{2\gamma -\frac{1}{2}}+\omega \left( \left\vert
y-z\right\vert \right) \leq 2\omega \left( \delta \right) +C_{\gamma }\delta
^{2\gamma -\frac{1}{2}},
\end{eqnarray*}%
which completes the proof of Step eight since for $\frac{1}{4}<\gamma <\frac{%
1}{2}$, we have $2\omega \left( \delta \right) +C_{\gamma }\delta ^{2\gamma -%
\frac{1}{2}}\leq C^{\prime }\omega \left( \delta \right) $ for $\delta >0$
sufficiently small. The proof of the Trace Method Theorem \ref{trace method
theorem}\ is now complete.


\begin{thebibliography}{KoRiSaSh2}
\bibitem[CaVa]{CaVa} \textsc{Luis A. Caffarelli and Alexis F. Vasseur,} 
\textit{The De Giorgi method for regularity of solutions of elliptic
equations and its applications to fluid dynamics,} Discrete and Continuous
Dynamical Systems - Series S, \textbf{3}(3), 409-427. DOI:
10.3934/dcdss.2010.3.409.

\bibitem[Chr]{Chr} {\textsc{M. Christ, }}\textit{Hypoellipticity in the
infinitely degenerate regime, Complex Analysis and Geometry, Ohio State
Univ. Math. Res. Instl Publ. \textbf{9}, }Walter de Gruyter, New York
(2001), 59-84\textit{.}

\bibitem[Fe]{Fe} {\textsc{V. S. Fedi\u{\i},} \textsc{\ }\textit{On a
criterion for hypoellipticity},\textit{\ }Math. USSR Sbornik\textit{\ }%
\textbf{14} (1971), 15-45.}

\bibitem[Fri]{Fri} \textsc{K. O. Friedrichs}, \textit{On the
differentiability of the solutions of linear elliptic differential equations}%
, Comm. Pure Appl. Math. \textbf{6} (1953), 299-326.

\bibitem[GaNh]{GaNh} \textsc{N. Garofalo and D. M. Nhieu,} \textit{Lipschitz
continuity, global smooth approximations and extension theorems for Sobolev
functions in Carnot-Carath\'{e}odory spaces}, J. Anal. Math. \textbf{74}
(1998), 67--97.

\bibitem[GiTr]{GiTr} {\textsc{D. Gilbarg and N. Trudinger,} \textsc{\ }%
\textit{Elliptic Partial Differential Equations of Second Order},\textit{\ }
revised 3rd printing, 1998, Springer-Verlag.}

\bibitem[Koh]{Koh} \textsc{J. J. Kohn}, \textit{Hypoellipticity of Some
Degenerate Subelliptic Operators}, Journal of Functional Analysis \textbf{159%
} (1998), 203-216.

\bibitem[KoRiSaSh]{KoRiSaSh} \textsc{L. Korobenko, C. Rios, E. Sawyer and
Ruipeng Shen,} \textit{Local boundedness, maximum principles, and continuity
of solutions to infinitely degenerate elliptic equations}, to appear in Mem.
Amer. Math. Soc., includes some of the content of \texttt{arXiv:1506.09203v5}%
, \texttt{arXiv:1506.01630v1} and \texttt{arXiv:1703.00774v1.}

\bibitem[KoRiSaSh2]{KoRiSaSh2} \textsc{L. Korobenko, C. Rios, E. Sawyer and
Ruipeng Shen,} \textit{Local boundedness, maximum principles, and continuity
of solutions to infinitely degenerate elliptic equations}, \texttt{%
arXiv:1506.09203v5.}

\bibitem[KuStr]{KuStr} {\textsc{S. Kusuoka and D. Stroock,} \textsc{\ }%
\textit{Applications of the Malliavin Calculus II}, J. Fac. Sci. Univ. Tokyo 
\textbf{\ 32} (1985), 1--76.}

\bibitem[Mor]{Mor} {\textsc{Y. Morimoto,} \textsc{\ }\textit{%
Non-Hypoellipticity for Degenerate Elliptic Operators}, Publ. RIMS, Kyoto
Univ. \textbf{22} (1986), 25-30; and \textit{Erratum to "Non-Hypoellipticity
for Degenerate Elliptic Operators",} Publ. RIMS, Kyoto Univ. \textbf{34}
(1994), 533-534.}

\bibitem[SaW3]{SaW3} {\textsc{E. Sawyer and R. Wheeden,} \textsc{\ }\textit{%
\ Degenerate Sobolev spaces and regularity of subelliptic equations}, Trans.
Amer. Math. Soc. \textbf{362} (2009), 1869 - 1906.}

\bibitem[Ste]{Ste} \textsc{Stein, Elias M.}, \textit{Singular Integrals and
Differentiability Properties of Functions}, Princeton University Press,
Princeton, N. J., 1970.
\end{thebibliography}
\end{document}